\documentclass[11pt]{article}
\usepackage[margin=1in]{geometry}
\usepackage{amsthm,amsmath,amsfonts,amssymb,mathtools}
\usepackage[numbers]{natbib}
\usepackage{enumitem}
\usepackage{booktabs}
\usepackage{tikz}
\usetikzlibrary{positioning,arrows.meta,calc}
\usepackage[colorlinks,citecolor=blue,urlcolor=blue]{hyperref}

\newcommand{\NN}{\mathbb{N}}
\newcommand{\PP}{\mathbb{P}}
\newcommand{\EE}{\mathbb{E}}
\newcommand{\RR}{\mathbb{R}}
\newcommand{\ZZ}{\mathbb{Z}}

\DeclareMathOperator{\Exp}{Exp}

\DeclareMathOperator{\GEM}{GEM}

\theoremstyle{plain}
\newtheorem{theorem}{Theorem}[section]
\newtheorem{proposition}[theorem]{Proposition}
\newtheorem{lemma}[theorem]{Lemma}
\newtheorem{corollary}[theorem]{Corollary}

\theoremstyle{definition}
\newtheorem{definition}[theorem]{Definition}
\newtheorem{example}[theorem]{Example}
\newtheorem{remark}[theorem]{Remark}

\title{The Dual-Population Benchmark Model}
\author{Alexander Gnedin\\
\small School of Mathematical Sciences, Queen Mary University of London\\
\small \texttt{a.gnedin@qmul.ac.uk}}
\date{}

\begin{document}
\maketitle

\begin{abstract}
Motivated by Fisher's and Hill's \cite{Fisher,Hill} ideas of ranking and pivotal quantities,
we introduce a  planar Poisson process (PPP) model in which the negative component \(\Pi_-\), acted upon by a choice operator {\rm C}, supplies a set of past benchmarks that divide the future points of the positive component \(\Pi_+\) into ranked categories. The benchmarks act as separators in an ordered paintbox while simultaneously acquiring the dual role of past standards established by the choice operator. The exceptional homogeneity properties of the PPP offer an infinitude of possibilities which absorb  many existing combinatorial structures
and enrich the toolbox of Bayesian distribution-free inference. A particular benchmark
generating mechanism considered here (the exponential race) amounts to the
 device of splitting into spacings of  nonhomogeneous order statistics.

On the methodological side,  the paper aims to highlight the role of order  as important characteristic of an exchangeable structure, complementary to the description in terms
of the components size. Moreover, we  advocate the viewpoint  that the order induced by a latent strength parameter is {\it intrinsically} inherited
from the distant-past temporal sampling  order, hence  the decoupled orders may  coexist within the framework of population duality without disturbing
size-biasedness of components and the full exchangeability within the sample.
This implies that the indistinguishability of components in the nonlinear CRP, sometimes regarded as nonexchangeability, does not in fact destroy exchangeability
and should be reconciled with the arrival ordering within the paradigm of  ordered structures.
\end{abstract}



\section{Introduction}
\label{sec:intro}

Motivated by Fisher's and Hill's \cite{Fisher,Hill} ideas of ranking and pivotal quantities -- Fisher's fiducial thinking about ranks and order statistics providing the historical root, Hill's treatment of the ranks of continuous observations as pivotal, uniformly distributed independently of the unknown continuous law, the explicit technical development -- we introduce the dual-population benchmark model of this paper.

The boundary theory of Markov processes offers two views on the evolution. By the forward view, the initial state $A$ of the process directs, by the law of large numbers,
 the path to a terminal state $W$ on the boundary. By backward view, the destination $W$ is fixed and ties  the path by the terminal condition. The classic example is the 
exchangeable urn process  with some initial composition $A$, and $W$ appearing at the end of the journey as the vector of frequencies. Other way round, conditioning on $W$,
which acts now as parameter, outputs a process of independent   draws.

In this paper we take an intermediate  relativistic view on  population processes,   in which the present time follows the infinite past and precedes the infinite future. 
The past process has reached the present time and revealed all parameters
 \({W}\), that  act as a matrix  impacting the future population,
 that  remains exchangeable within  its own members. The background idea of the dual model to follow is that the past  choice operator 
created benchmarks which shape both size and ordering of the future generations through testing against the latent benchmarks.

Taking distribution-free, combinatorial view,   we will focus on the central   infinite-component structure: the random partition of \(\mathbb{Z}_{>0}\). The same distribution-free stance underlies much of statistical learning and pattern recognition practice \cite{merkov2011}, where a population is classified before its generating law is known. The source of the infinite-component 
exchangeability is the factorial code of a permutation, introduced in terms of relative ranks by Laisant \cite{Laisant1888}
and unlocked to infinity by Rubin in a short abstract on the secretary problem \cite{Rubin1966}. 
In  terms of the cycle insertion the code elevated
to the  virtual permutations of the representation theory of the infinite symmetric group ${\mathfrak S}_\infty$ \cite{BorodinOlshanski}.
For fixed \(n\) these twin codes are related by the fundamental bijection that splits total ranking row
at records and closes the broken parts in cycles, 
   but in the limit they completely disentangle.
Cycle insertion highlights the masses and is intrinsic to alphabet-symmetry: equal-size components are indistinguishable. The resulting preferential attachment has become  mainstream in the
work on distribution-free. The former interest in  the rank code \cite{Hill}
has drifted away. The dual-population model attempts to bring it back.
Retrospectively, the  historic development shows many lost opportunities;
 for instance, should Bayes admit that the predictive probability of the next failure after \(n-k\) successes equals \(2^k/(n+1)!\), we would have learned about the Karamata--Stirling 
trials and the Ewens sampling formula much earlier.

In the foundational paper Pitman \cite{PitmanPTRF} observed that arranging the  frequencies in the order of starters brings enormous analytic simplification in the representation
of the probability function as moments of the infinite-dimensional frequency vector, replacing infinite symmetric series inherent to Kingman's representation. 
Following this,   Kerov \cite{KerovSubordinators}, interpreted the connection as transition from Kingman branching of partitions to   the linear comb graph of Stanley \cite{StanleyFibonacciLattice}.
In the dual population picture the spine of temporally ranked ancestors stays in the past, affecting 
division of the future population in ranked categories, whose members  occupy their place 
in the new hierarchy by testing their standard latent variable against the inhomogeneous hierarchy of benchmarks.
The ranking mechanism is the factorial code of the permutation, which itself is a detailed version of the Law of Succession.

Exchangeability of ordered structures appears under different guises in many ranking procedures,
in extreme value-theory, sequential choice \cite{secretaryProblem},
homogeneous  and inhomogeneous stick-breaking
 \cite{  PitmanTran2015,   iksanov2016renewal, GnedinPitman2005, GnedinIksanovMarynych2010},
population models \cite{mailler}, Bayesian distribution-free inference
\cite{LijoiPruenster2010} and machine learning,\cite{pouxmedard2021powered}.
The PPP framework accommodates many of the familiar  models naturally and offers new. 
The model is distribution-free in the sense that any product of a finite measure and Lebesgue measure yields an equivalent ordering structure; we work throughout with exponential strengths merely for analytic convenience.  We may or may not  remain inside the gamma--beta universe, but in any case the algebra shifts from addition to comparison.

The analytical and combinatorial elements we use are mostly well known in combinatorial stochastic processes \cite{Pitman2006}, extreme-value theory \cite{resnick1987extreme}, algebraic combinatorics  \cite{StanleyEC1}, Bayesian inference and    social-choice theory \cite{Aleskerov}, although some elements may be new. A substantial element where we step away  from the tradition is that we 
assume   a realised scattered population together with a mechanism that selects distinguished past representatives by a social-choice operator. These representatives act as benchmark setters against which a future population is compared. The emphasis therefore shifts from updating frequencies to updating residual strengths of benchmarks.




The benchmark model begins with a hierarchy extracted from the past and tests future observations against this hierarchy. 
The model at this stage is neutral to the orientation of the order: reversing the strength scale is interpreted as the same structure, though shuffles can be readily introduced \cite{gnedin2006coherent}.
We assume that past and future populations are two interacting particle systems whose principal mode of interaction is comparison against a hierarchy of benchmarks extracted from the past by a social-choice operator. Each future particle is ordered \emph{locally} and the lowest-rank element is delegated to tests against the benchmarks (a construction common in the best-choice problem of optimal stopping \cite{gnedin2026memoryless}). The outcome updates both the rank of the particle and the residual strengths available to subsequent particles. The interaction is therefore neither collision, repulsion nor reinforcement, but pure comparison. The resulting sampled ordered structure is the primary observable; the exchangeable  {\it ordered}  partition (i.e. {\it composition structure})  is recovered in some form representing the ordered paintbox 
in the infinite-sample limit only.

\section{Preliminaries}
\label{sec:prelim}
\subsection{Boundary theory of Markov chains through The Law of Succession}

The boundary theory of Markov processes starts with, and is most simply illustrated by, the Bayes--Laplace random walk, reintroduced by Markov (1917) and P\'olya (1923) and widely known as the P\'olya urn.

The walk moves on the Pascal graph \(\ZZ_{\ge0}^2\), and its state \((a_1,a_2)\) at time \(n=a_1+a_2\) also carries the odds for the next move. The dynamics is additive: a step adds one of the two basis vectors, \((a_1,a_2)\mapsto(a_1,a_2)+e_1\) or \(+e_2\). Odds are a separate, projective view of the same lattice point: \((a_1,a_2)\) as odds names the ray through it, so that \((a_1,a_2)\) and \((2a_1,2a_2)\) are the same odds though different states of the walk. The object \(\boldsymbol0\) is the walk started at \((0,0)\), the symmetric Bernoulli walk fluctuating about the diagonal; starting instead at \((w_1,w_2)\) gives the Markov--P\'olya two-colour urn, fluctuating about the ray through \((w_1,w_2)\).

Exchangeability in this picture appears through Vershik-Kerov \cite{BorodinOlshanski} centrality or Dynkin  
sufficiency \cite{DiaconisFridmanSufficiencyIndianInst};
given a point of the discrete simplex  
$$
D_n=\{(a_1,a_2):a_1,a_2\ge0,\ a_1+a_2=n\}$$ is reached by the walk, the history is forgotten, that is, in retrospect
all \(\binom{a_1+a_2}{a_1}\) paths connecting $(a_1,a_2)$ to the root $(0,0)$ are equally likely; thus $D_n$ parametrises elementary $n$-step walks. The boundary of the graph, the weak closure of \(\bigcup_nD_n\), was identified, from different viewpoints, by Hausdorff and de Finetti with the unit interval. 
$$\Delta_1=\{(w_1,w_2)\in [0,1]^2: w_1+w_2=1\}.$$

Started at \((a_1,a_2)\), the walk converges almost surely -- by martingale convergence, the odds process being itself a bounded martingale -- to a random limit \(W\sim\mathrm{Beta}(a_1,a_2)\) on the boundary. The limit is genuinely random: with probability one \((a_1,a_2)\) is not proportional to \(W\), the sole exception being the start \((0,0)\), which has no ray of its own to compare against and whose limit is the
symmetric-in-the-odds nonrandom success probability
 \(1/2\), reached by the Bernoulli law of large numbers.

Looking backward, conditionally on the limit being proportional to normalised \((w_1,w_2)\), the walk is an i.i.d.\ Bernoulli sequence with parameter $w_1$ -- de Finetti's theorem in this coordinate system. This is the sense in which the boundary state of Bayes-Laplace walk is alphabet-asymmetric: the walk's dynamics never distinguished the two boxes, but its terminal state, conditioned on, produces a biased coin distinguishing them by frequency.

The duality carries forward as well as back. A boundary point \((w_1',w_2')\), not necessarily rational, becomes the initial state of a walk into the future, staying in the cone 
$$(w_1',w_2')+\ZZ_+^2$$ and following the same Law of Succession that produced it. The present is the state in which the past has already shaped the parameters of the particle being observed. The planar Poisson process model realises this with a finite, rather than infinite, past: a pure model when a single benchmark's parameters are known outright, a mixed model when only a group of benchmarks is known and which of them prevails in a comparison among themselves -- a tournament, in the sense of Section~\ref{sec:insertion} -- is not observed.

The uniform law on \(D_n\) -- an equal prior weight on every one of the \(n+1\) odds \(0:n,\,1:(n-1),\,\dots,\,n:0\) -- is the right starting assumption for a reason internal to alphabet symmetry itself, not an extraneous convenience. Among these odds there is no apriori reason to prefer one to another: this is exactly the assertion that the box-permutation group acts on our judgment the same way it acts on the state space, alphabet symmetry stated as an epistemic indifference rather than as a fact about dynamics. If this indifference is granted as the null hypothesis, the Law of Succession follows for the next observation without further assumption, by conditioning the uniform law on \(D_n\) on the observed outcomes and passing \(n\to\infty\): no separate device, such as Laplace's own argument by elimination over sequences of heads and tails, is required. The indifference is asserted entirely among lattice points of \(D_n\), with no ray or coordinate chosen at any stage; the projective object \(\Delta_1\) only appears afterward, as the limit of an already-fixed discrete construction, so the usual difficulty attached to a ``uniform'' continuous prior -- that uniformity depends on a choice of parametrisation -- does not arise.
\subsection{Two symmetric groups: balls and boxes}

Two distinct symmetric groups act on a partition of a population into types, on two distinct index sets. The group \(\mathfrak S_\infty\) acts on the population units -- the balls -- permuting which unit carries which value; invariance under this action is exchangeability in the classical, de Finetti sense. A second, independent group acts on the types -- the boxes, or tables in the language of the Chinese restaurant process -- permuting box names among themselves. Invariance under this second action is alphabet symmetry, in the information-theoretic sense that a box name carries no content on its own: entropy, counts, and every quantity of interest should be blind to which box is called what.

The two groups act on different sets, and neither is a special case of the other. Alphabet symmetry, stated this way, has nothing to do with ranking a priori: it is a statement about the box-permutation group, independent of any order imposed on the boxes. Ranking is what discharges it. Once a rule assigns each configuration a canonical order on its boxes -- discovery order, size order, or any rule intrinsic to the configuration -- a box permutation that would disturb this order is no longer a symmetry of the ranked object. Alphabet symmetry is used up in selecting the canonical representative of each orbit; this is a combinatorial move, choosing one representative per orbit, and carries no probabilistic content. What extensions of \(\mathrm{EPY}(\alpha,\theta)\) that tilt positions in the hierarchy \cite{GnedinPartitionsConstraints, GnedinBiasRecord, Favaro, mailler, LijoiPruenster2010} disturb is only this second symmetry, once a ranking has been fixed; the underlying sample remains fully exchangeable in the first, classical sense throughout. Ball exchangeability cannot be discharged by any choice of representative: the sample either arose from an exchangeable mechanism or it did not, and no relabelling convention manufactures the property after the fact.
The alphabet-symmetry of boxes, if desired, can be gained by the uniform symmetrisation, which forgets the order, and typically replaces a nice factored distribution formula by a sum over the suitable symmetric group.

We iterate the above in the form of  matrix of the relation  `partition of balls in boxes'.
Note that
partition of ${\mathbb Z}_{\geq 0}$  is an equivalence relation, best representable  by an incidence matrix with coordinates $({\rm ball, box})$, which is most usefully 
 randomised by  Rubin's marking device which multiplies each row by an independent uniform random variable, latent mark.  The full group acting on random partitions is the bisymmetric group ${\mathfrak S}_\infty\times {\mathfrak S}_\infty$ (equivalent to Gelfand pair see \cite{BorodinOlshanski}). Balls exchangeability is invariance under shuffling the rows, boxes alphabet-symmetry -- columns; in both cases it is sufficient to require invariance of distribution under adjacent transpositions only. 
Classification of  extreme (ergodic, indecomposable)  ({\it per definition bisymmetric}) partitions is the Kingman paintbox representation by a boundary point on Kingman's simplex.

Now, introducing an independent of marking {\it weak} order and leaving only
 row exchangeability results in a much larger class of {\it ordered} ball-exchangeable partitions classified by elements of  the space of closed subsets of $K\subset [0,1]$, whose complementary opens sets $[0,1]\setminus K$ provide `percentile intervals'  for blocks with positive frequency, while $K$ itself is a source 
of singleton dust: these are the ordered paintboxes from
 \cite{GnedinRepresentationofCompositionStructures}, alternatively characterised by the class of quasi-uniform distributions on $[0,1]$ that appear as probability integral transform of the general probability distribution on 
${\mathbb R}$. Every continuous distribution is mapped to ${\rm U}[0,1]$ and will output sampling dust
comprising Rubin's infinite permutation,
while random paths of subordinators or NTR ditributions will induce homogeneous or nonhomogeneous stick-breaking ordered partitions.

An equivalent natural choice of the ground measure is ${\rm Exp}(1)$, 
the convention adopted in the exponential race to follow; where the benchmarks receive latent variables
of the kind $E/w$.
The future population is exchangeable in the ball sense: its law is invariant under \(\mathfrak S_\infty\) acting on individuals, and after ranking, components appear in the size-biased order of nonhomogeneous order statistics. What is sometimes read in the statistical literature as nonexchangeability of such constructions is a disturbance of alphabet symmetry once tables are ranked by arrival, not a disturbance of ball exchangeability -- the ranking of tables should be updated each time a new representative, arriving at time \(n\), tests against the benchmarks differently from how the first \(n-1\) did. Alphabet symmetry is regained formally by a final shuffle of the tables under the box group.

\subsection{Ranking conventions}

Order in the distribution-free framework suggests peculiar views on familiar concepts, developed systematically in Section~\ref{sec:algebra}. Two remarks fix the analytic conventions used throughout.

Ranking occurs by comparing latent variables of the population units. For the sampled members these are i.i.d.; for the benchmarks, in general, not. The expected value of a latent variable acquires the meaning of a measure of relative strength -- the ability to rank lower or higher in comparisons. The direction of ranking is fixed for convenience only and carries no interpretation (though we speak of strength and winners). Recalling integration by parts, for an independent random variable \(Y\ge0\) and a uniform reference \(U\stackrel{\rm d}{=}\mathrm{U}[0,1]\),
\[
\EE Y = \PP[Y>U].
\]
Higher binomial moments are interpreted similarly. This is the fiducial device of Fisher and Hill \cite{Fisher,Hill} in its minimal form: \(U\) is a pivotal quantity, its law carrying no information about \(Y\)'s unknown distribution, and the inversion -- known pivot in, unknown-law feature out -- recovers \(\EE Y\) same way as a rank recovers Hill's \(A_n\) statistic distribution-free of the underlying law.

The PPP setting involves random variables supported on all of \(\RR_+\), unbounded, so the reference above must be adapted: no single bounded pivot can test \(X\) at every scale, and the fiducial comparison becomes a family of pivots indexed by rate rather than one fixed \(U\). In that case the probability of ranking higher than a ``uniform variable on \(\RR_+\)'' relates to the mean via regularisation,
\[
\EE X = \int_0^\infty \PP[X>x]\,{\rm d}x = \lim_{\delta\to0}\PP[X>E/\delta]/\delta,
\]
where \(E\) is a standard exponential random variable acting as a scaling benchmark: this is the mechanism by which the benchmarks of Section~\ref{sec:setup} extract a comparison scale from an otherwise translation-invariant strength axis. More generally, the regularised probability that \(X\) survives one independent ``uniform on \([0,\infty)\)'' of intensity \(w_1\) and fails against a second of intensity \(w_2\) is
\[
\lim_{\delta\to0}\PP\bigl[X>E_1/(w_1\delta),\ X<E_2/(w_2\delta)\bigr]/\delta
=w_1\,\EE X\,,
\]
from \(\PP[X>E_1/(w_1\delta),X<E_2/(w_2\delta)]=\EE\bigl[(1-e^{-w_1\delta X})e^{-w_2\delta X}\bigr]\sim w_1\delta\,\EE X\) as \(\delta\to0\), which recovers the elementary size-biased pick (survive one uniform) and the anti-size-biased deletion (fail against one uniform) that appear in the deletion kernels of \cite{GnedinHaulkPitman2010}.

\section{The geometric setup}
\label{sec:setup}
 
Let \(\Pi\) be a homogeneous planar Poisson point process of unit intensity on \(\RR\times\RR_+\). Each atom \((t,x)\) carries a birth time \(t\) and a latent strength \(x\). Split the process according to the sign of time:
\begin{align*}
\Pi_- &= \Pi\cap\bigl((-\infty,0)\times\RR_+\bigr)\qquad\text{(past population)},\\
\Pi_+ &= \Pi\cap\bigl((0,\infty)\times\RR_+\bigr)\qquad\text{(future population)}.
\end{align*}
The two populations are independent. The past generates benchmarks; the future is ranked relative to those benchmarks.
 
For an interval \(I\subset\RR\) write \(\Pi_I=\Pi\cap(I\times\RR_+)\). The minimal strength \(M(I)\) in \(\Pi_I\) is distributed as \(\Exp(|I|)\) and is independent across disjoint intervals.
 
Minima with larger weight tend to be ranked lower (closer to the origin).
 
\begin{definition}[Social choice operator]
\label{def:choice}
A social choice operator {\rm C}, or simply {\it choice} is a measurable map that sends a configuration \(A\subset\Pi_-\) to a (finite or countable) subset \({\rm C}(A)\subset A\). A global choice operator satisfies \({\rm C}(A)={\rm C}(\Pi_-)\cap A\) and automatically obeys path-independence.
\end{definition}
Per default in this paper the social choice operator is global.

\begin{example}
\begin{enumerate}[label=(\roman*)]
\item For a system of disjoint intervals \(I_i\) the set of minima \(M(I_i)\) consists of the exponential-race winners associated with those intervals.
\item A point \((t,x)\in\Pi_-\) is a south-east record if the rectangle \([t,0]\times[0,x)\) is empty. The projected heights of the south-east records form a scale-invariant Poisson process of intensity \(1/x\,{\rm d}x\). The set of such records is Pareto boundary.
\end{enumerate}
\end{example}
The sampling operator \({\rm S}(\theta)\) extracts, from each successive interval $[\theta(n-1),\theta n]$ of length \(\theta\) in the positive half-plane, the point of lowest strength,
that is ranked $1$ on the strength scale in the population above the time slot; this sample element is regarded as representative elected to represent the local population.
The resulting sample strengths are i.i.d.\ \(\Exp(\theta)\).
 
\section{Fixed points, dust, multinomial sampling}
\label{sec:dust}
 
The PPP is {\it marginally simple}, that is has zero probability to occupy a line, neither on the time scale nor on strength. 
Thus to create absolute standards above zero we need to use conditioning, which is not in the province of the choice operator $\rm C$. However, fixed points can be considered as degenerate dust.
 
Let \(J=[a,b]\subset\RR_+\) be an interval on the strength scale. Consider the trinomial choice function \({\rm C}(\Pi_-)=\Pi_-\cap(\RR_-\times J)\), interpreted as interval or fuzzy testing. This divides the future population into three strength categories: below \(J\), within \(J\), above \(J\). Since the strength values of \(J\)-benchmarks are dense in \(J\), ranking of the \(\Pi_+\) population within \(J\) produces dust of singletons.
 In Kingman's theory this diffuse component is commonly considered a special category, which in exchangeable coalescent processes appears as a source of new progenitors, and in fragmentation processes as a terminal state. From the viewpoint of ordered partitions (or unlabelled composition structures) this is a legitimate splitter.
 Testing against \(J\)-benchmarks generates a trinomial partition of \(\Pi_+\) with frequencies coming from the percentiles of the exponential distribution. Bernoulli splitting appears then as a limit as \(J\) degenerates to a single point \(x\); no conditioning on an atom of fixed strength is required.
 
Alternatively, we may fix a distinguished benchmark, namely the latest arrival in the past population with strength in \(J\). Letting \(|J|\) shrink to a point \(x\) produces a Bernoulli partition of \(\Pi_+\)
by  independent binary marking.

Choosing arbitrary domain in the past and projecting this restricted Poisson scatter in the future yields the  splitting by points of univariate Poisson process of arbitrary intensity, equivalently providing the most general structure of hazard rates.

This construction can be iterated in many ways. Ranking against nested interval hierarchies produces Cantor sets, characteristic of the Poisson--Kingman partitions \cite{Pitman2006} or regenerative composition structures
\cite{GnedinPitman2005}.
 
Furthermore, setting \({\rm C}(\Pi_-)=\Pi_-\cap(I\times J)\), with \(I\) an interval on the temporal scale, restricts the choice to a finite Poisson scatter of benchmarks in the strength band \(J\) arriving during \(I\), hence partitions the future population into a random number of categories -- a natural combinatorial mixing of the binary (interval) test above with the fuzzy splitter.
 

\section{Splitting constructions}
This section employs elementary examples to illustrate the simple transition from splitting $[0,1]$ by nonhomogeneous uniform order statistics (boxes) and dropping further $n$ uniform points (balls) to splitting
$[0,\infty)$ by $n$ nonhomogeneous exponential order statistics (benchmarks), dropping further exponentials (strengths) and counting combinatorial residuals at splits. 
This is most transparent when tilting occurs on the combinatorial level, with parameters derived from the initial configuration.
The term {\it composition} of integer $n$ means here an integer vector $(a_1,\ldots,a_k)$ whose total is $n$ and the parts may be $0$.

\subsection{Binary splitting}
\subsubsection{One benchmark and the two-colour Markov-P\'olya urn}

\label{sec:one-benchmark}
We start with the iconic example going back to Bayes and Laplace.
Partition the past population into successive unit-length time slots and let the choice operator retain only the minimum of a single fixed slot. Call this minimum \(B\stackrel{\rm d}{=}\Exp(1)\). A future sample of size \(n\) generated by \({\rm S}\) consists of \(n\) i.i.d.\ \(\Exp(1)\) strengths independent of \(B\).
 
The benchmark divides the sample into two types: ``below'' (strength \(<B\)) and ``above'' (strength \(>B\)). Let \(A_1\) (resp.\ \(A_2\)) be the corresponding counts, so \(A_1+A_2=n\). Because the \(n+1\) variables are i.i.d.\ continuous, the rank of \(B\) among them is uniform on \(\{1,\dots,n+1\}\). Consequently every composition is equally likely:
\begin{equation}
\label{eq:uniform}
\PP[A_1=a_1,A_2=a_2]=\frac1{n+1},\qquad a_1+a_2=n,
\end{equation}
which is 
the step $n$ occupation law of a two-colour Markov--P\'olya urn that starts with one ball of each colour.
 
\subsection{Two past slots, merged benchmark}
\label{sec:two-slots}
 
Retain the minima \(B_1,B_2\) of two distinct unit slots and merge them into the single benchmark \(B=\min(B_1,B_2)\). Then \(B\stackrel{\rm d}{=}\Exp(2)\). Averaging the binomial likelihood against the \(\Exp(2)\) density yields
for $n=a_1, a_2$

\begin{equation}
\label{eq:merged-unit}
\PP[A_1=a_1,A_2=a_2]=\frac{2(a_2+1)}{(a_1+a_2+1)(a_1+a_2+2)}.
\end{equation}
This is the step $n$ law of a P\'olya urn started with one ball of colour ``below'' and two balls of colour ``above''.
\subsection{Two distinct weights, merged benchmark}
The next formula  is again standard
Let \(B=\min(B_1,B_2)\) with \(B_i\stackrel{\rm d}{=}\Exp(w_i)\), \(w_1,w_2>0\). Then \(B\stackrel{\rm d}{=}\Exp(s)\), \(s=w_1+w_2\). For a future sample of size \(n\) generated by \({\rm S}\),
\begin{equation}
\label{eq:merged-weighted}
\PP[A_1=a_1]=\binom{n}{a_1}\frac{{\rm B}(1+a_1,s+n-a_1)}{{\rm B}(1,s)},\qquad a_1=0,\dots,n.
\end{equation}

 
Only the total weight \(s\) appears; the split between the two slots is invisible after merging.
 

\subsection{Non-binary splitting}
\subsubsection{Two separate benchmarks: joint composition}
\label{sec:two-benchmarks}
 
Keep the two weighted benchmarks \(B_1\stackrel{\rm d}{=}\Exp(w_1)\) and \(B_2\stackrel{\rm d}{=}\Exp(w_2)\) unmerged. They divide \((0,\infty)\) into three cells. Write \(A_0,A_1,A_2\) for the occupation numbers of a future sample of size \(n\), and set \(s=w_1+w_2\).
 
\begin{proposition}[Joint composition law]
\label{thm:joint-two}
\begin{eqnarray*}
\PP[A_0=a_0,A_1=a_1,A_2=a_2]
=\\
\binom{n}{a_0,a_1,a_2}w_1 w_2\,B(a_0+1,s+n-a_0)
\bigl[B(a_1+1,a_2+w_1)+B(a_1+1,a_2+w_2)\bigr].
\end{eqnarray*}
\end{proposition}
 
\begin{proof}
The \(n+2\) variables \(B_1,B_2,X_1,\dots,X_n\) are independent exponentials with rates 
\(w_1,w_2,1,\dots,1\). By memorylessness their arrival order can be generated sequentially: at each step, among the items not yet placed, the next to arrive (on the strength scale) is chosen with probability proportional to its rate, and the survivors remain independent exponentials with their original rates.
 
Condition on the event that \(B_1\) arrives before \(B_2\) (the complementary event is symmetric under \(w_1\leftrightarrow w_2\)). On this event the arrival pattern is: \(a_0\) future points, then \(B_1\), then \(a_1\) future points, then \(B_2\), then the remaining \(a_2\) future points. Multiplying the successive discovery probabilities and simplifying the resulting product of Gamma functions gives
\begin{eqnarray*}
\PP[B_1\text{ first},A_0=a_0,A_1=a_1,A_2=a_2]
=\\
\binom{n}{a_0,a_1,a_2}w_1 w_2\,B(a_0+1,s+n-a_0)B(a_1+1,a_2+w_2).
\end{eqnarray*}
 
The factor \(B(a_0+1,s+n-a_0)\) depends on \((w_1,w_2)\) only through \(s=w_1+w_2\), since it governs the placement of the \(a_0\) points before whichever of \(B_1,B_2\) arrives first, and \(B_{(1)}\stackrel{\rm d}{=}\Exp(s)\) regardless of which benchmark this turns out to be. Adding the symmetric ``\(B_2\) first'' term yields the claimed formula.
\end{proof}
 
Summing over \(a_1+a_2=n-a_0\) recovers the merged law \eqref{eq:merged-weighted}.

The joint law is the marginal law of a three-colour urn scheme: one ball for ``below \(B_{(1)}\)'' of weight \(s=w_1+w_2\)---bundling the two benchmarks together, since neither the future sample nor the other benchmark can distinguish \(B_1\) from \(B_2\) until one of them has actually been revealed---followed, once \(B_{(1)}\) is revealed, by a split of the remaining weight into \(w_1\) or \(w_2\) with probabilities \(w_1/s\) and \(w_2/s\) according to which benchmark arrived first.

\subsubsection{Matched rates and the Dirichlet--multinomial}
\label{sec:matched}
 This is  the example of $k$-colour  P{\'o}lya urn with starting composition $(1,\ldots,1)$, corresponding
to $k$ past benchmarks exchangeable with the  future sample 
\begin{proposition}[Matched-rate embedding]
\label{prop:matched}
Let \(B_1,\dots,B_k\) be i.i.d.\ \(\Exp(1)\) benchmarks (unmerged) and let the future sample be generated by \({\rm S}\). The resulting composition \((A_0,\dots,A_k)\) is uniformly distributed over the \(\binom{n+k}{k}\) compositions of \(n\) into \(k+1\) parts. Equivalently, it is $n$-state of a \((k+1)\)-colour P\'olya urn started with one ball of each colour (Dirichlet-multinomial\((1,\dots,1)\)).
\end{proposition}
 
\begin{proof}
All \(n+k\) variables are i.i.d.\ continuous, so every subset of size \(k\) is equally likely to be the set of ranks occupied by the benchmarks. The composition is a deterministic function of that subset and is therefore uniform.
\end{proof}

\begin{remark}
A combinatorial derivation of the joint distribution used only exchangeability: the original elements are exchangeable  with the sample, therefore ranking by strength yields a uniform permutation
of $k$ indistinguishable benchmarks and $n$ data points, which immediately yields the distribution of the composition. In the P{\'o}lya urn the exchangeability acts on the temporal scale and is much less obvious due to the sequential nature of the Markov chain. 
The matching of rates is essential: if the common benchmark weight differs from the future sampling rate the composition ceases to be uniform.
\end{remark}

\subsection{The exponential race with finite set of benchmarks}
\label{sec:race}

Let \(B_1,\dots,B_k\) have arbitrary positive weights \(w_1,\dots,w_k\),  and let the future sample \(X_1,\dots,X_n\) be i.i.d.\ \(\Exp(1)\), independent of the benchmarks. Write \(\sigma=(i_1,\dots,i_k)\) for the discovery order of the benchmarks and \(W_r(\sigma)=\sum_{s=r}^k w_{i_s}\) for the residual weight at stage \(r\).
The following observation builds on  ranking properties of the nonhomogeneous exponential order statistics.
\begin{proposition}[Size-biased order]
\label{prop:size-biased}
\[
\PP[\sigma=(i_1,\dots,i_k]]=\prod_{r=1}^k\frac{w_{i_r}}{W_r(\sigma)}.
\]
\end{proposition}
 
\begin{theorem}[Independent residual spacings]
\label{thm:spacings}
Conditional on \(\sigma\), the spacings \(D_r=B_{(r)}-B_{(r-1)}\) (with \(B_{(0)}=0\)) are independent and \(D_r\mid\sigma\stackrel{\rm d}{=}\Exp(W_r(\sigma))\).
\end{theorem}
 
Set \(V_r=e^{-D_r}\). Then \(V_r\mid\sigma\stackrel{\rm d}{=}{\rm B}(W_r(\sigma),1)\) are independent. The stick-breaking coordinates
\begin{align*}
P_1&=1-V_1, &
P_r&=(1-V_r)\prod_{j<r}V_j, &
P_{k+1}&=\prod_{j=1}^k V_j
\end{align*}
form a random partition of unity, and \(P_r\) is the conditional probability that a future point falls into the \(r\)-th cell.
\begin{theorem}[General composition law]
\label{thm:general}
For \(a_0+\cdots+a_k=n\),
\begin{eqnarray}\nonumber
\label{eq:general}
\PP[A_0=a_0,\dots,A_k=a_k]
=\\
\binom{n}{a_0,\dots,a_k}\Bigl(\prod_{i=1}^k w_i\Bigr)
\sum_{\sigma\in\mathfrak{S}_k}
\prod_{r=1}^k {\rm B}\bigl(A_r(\sigma)+a_r+\cdots+a_k,\,a_{r-1}+1\bigr).
\end{eqnarray}
\end{theorem}
 
\begin{proof}
Given \(\sigma\) the vector \(A\) is multinomial with parameters \(P_r\). Expanding the powers of the stick-breaking coordinates,
\[
\prod_{r=1}^{k+1}P_r^{a_{r-1}}=\prod_{j=1}^k(1-V_j)^{a_{j-1}}V_j^{a_j+\cdots+a_k}.
\]
By independence of the \(V_j\) and \(\EE[(1-V)^p V^q]=W\,{\rm B}(q+W,p+1)\) for \(V\stackrel{\rm d}{=}{\rm B}(W,1)\),
\[
\PP[A=a\mid\sigma]=\binom{n}{a_0,\dots,a_k}\prod_{j=1}^kW_j(\sigma)\,{\rm B}\bigl(W_j(\sigma)+a_j+\cdots+a_k,\,a_{j-1}+1\bigr).
\]
Multiplying by \(\PP[\sigma]=\prod_j w_{i_j}/W_j(\sigma)\) cancels the \(W_j(\sigma)\) factors and leaves \(\prod_i w_i\), independent of \(\sigma\). Summing over \(\sigma\in\mathfrak{S}_k\) produces \eqref{eq:general}.
\end{proof}

 

\section{The Dirichlet integral and quasisymmetric functions }
\label{sec:algebra}
In this section we write the Dirichlet integral in the survival variables \cite{PaisleyBleiJordan2012},
which on the combinatorial level amounts to manipulation with stick breaking factors.
In the limiting process  the ordered Kingman paintbox appears  in survival coordinates. 
The algebra of symmetric functions of preferential attachment is replaced by 
quasisymmetric.

\subsection{The Dirichlet integral in survival coordinates}

\label{sec:DI}
Consider the \(k\)-simplex 
$$\Delta_k=\{\boldsymbol s=(s_1,\ldots,s_k):1=s_0>s_1>\cdots>s_k>0\},$$ 
and the combinatorial analogue of a decreasing vector, a strictly decreasing integer sequence \(n=0>A_1>\cdots>A_k>0\), with \(s_{k+1}:=0\), \(A_{k+1}:=0\).

\begin{lemma}[Dirichlet integral in homogeneous survival coordinates with surface differential]
\label{lem:DI}
\begin{equation}
\label{eq:DI}
\Bigl(\prod_{i=1}^k\binom{A_{i-1}}{A_i}\Bigr)\int_{\Delta_k}\prod_{i=1}^{k+1}(s_{i-1}-s_i)^{A_{i-1}-A_i-1}\,{\rm d}\boldsymbol s
=\frac{A_0}{\prod_{i=1}^{k+1}(A_{i-1}-A_i)}.
\end{equation}
\end{lemma}

\begin{proof}
Write \(a_i=A_{i-1}-A_i>0\) for \(i=1,\ldots,k+1\), so \(\sum_i a_i=A_0=n\). The integral on the left is the standard Dirichlet integral over the ordered simplex \(\Delta_k\), equal to \(\prod_i\Gamma(a_i)/\Gamma(A_0)=\prod_i(a_i-1)!/(A_0-1)!\). Using \(\binom{A_{i-1}}{A_i}=A_{i-1}!/(A_i!\,\mu_i!)\), the product \(\prod_{i=1}^k\binom{A_{i-1}}{A_i}\) telescopes to \(A_0!/(A_k!\prod_{i=1}^k a_i!)\). Multiplying the two displays,
\[
\frac{A_0!}{A_k!\prod_{i=1}^k a_i!}\cdot\frac{\prod_{i=1}^k(a_i-1)!\,(A_k-1)!}{(A_0-1)!}
=A_0\cdot\frac1{A_k}\cdot\prod_{i=1}^k\frac1{a_i}
=\frac{ A_0} {\prod_{i=1}^{k+1}a_i},
\]
using \(a_{k+1}=A_k-A_{k+1}=A_k\); this is the claimed right-hand side.
\end{proof}

The domain \(\Delta_k\) is a single chamber of the full weight simplex \(\{s_i\ge0,\sum s_i\le1\}\), singled out by the genuine order \(s_1>\cdots>s_k\) of the benchmark hierarchy. Summing the integrand of \eqref{eq:DI} over all \(k!\) chambers (equivalently, over all orderings of the parts \(\mu_i\)) recovers the fully symmetric classical Dirichlet integral. Restricting to the one chamber forced by the hierarchy's actual order is exactly what breaks full symmetry down to \emph{quasi}-symmetry: the integrand, and hence \eqref{eq:DI}, is invariant only under reindexings that preserve the order of the \(a_i\) -- the defining symmetry of quasisymmetric functions, taken up below.

\begin{remark}[Connection to the general composition law]
\label{rem:DI-general}
Lemma~\ref{lem:DI} is the survival-coordinate form of the computation underlying Theorem~\ref{thm:general}: the identity 
$$\EE[(1-V)^pV^q]=w {\rm B}(q+w,p+1)$$ 
for \(V\stackrel{\rm d}{=}{\rm B}(w,1)\),
applied along one fixed discovery order \(\sigma\), is an
 instance of \eqref{eq:DI} with the \(W_i\) identified with the residual comparison counts \(W_i(\sigma)+a_i+\cdots+a_k\) of that proof. Theorem~\ref{thm:general}'s sum over \(\sigma\in\mathfrak S_k\) is  the statement that a \emph{general} weighted race integrates over all \(k!\) chambers, not the single ordered one of \(\Delta_k\); Section~\ref{sec:regenerative} identifies exactly when this sum collapses back to the single-chamber form \eqref{eq:DI}.
\end{remark}

The classic predicting rule in these combinatorial coordinates becomes a splitting/insertion rule: a new point finds the gap between \(A_r\) and \(A_{r-1}\) with probability \((A_{r-1}-A_r)/A_0\) and settles there at random, creating the transition
\begin{equation}
(A_0,A_1,\ldots, A_k) \longrightarrow (A_0+1, \ldots, A_{r-1}+1, \widetilde{A},  A_r,\ldots,A_{k+1}),
\end{equation}
where the newly generated intermediate level \(\widetilde{A}\) is distributed uniformly over the discrete range \(\{A_r+1,\ldots,A_{r-1}\}\).

Starting from strengths \(X\geq0\), apply the probability integral transform. For exponential strengths one may write \(U=e^{-X}\). More generally, \(U=1-F(X)\). The benchmark hierarchy becomes an ordered sequence
$1>s_1>s_2>\cdots>0,$
which is a convenient distribution-free coordinate system.

Let \(a=(a_1,a_2,\dots)\) record occupancies between neighbouring benchmark levels. Define tail counts
\[
A_i=a_i+a_{i+1}+\cdots.
\]
Then \(a_i=A_i-A_{i+1}\). The quantities \(A_i\) count future observations exceeding benchmark level \(s_i\). They are therefore comparison counts rather than occupancy counts.

\begin{definition}[Dimension formula]
\label{def:dimension}
Let \(A_1>A_2>\cdots>A_d>0\), \(A_{d+1}=0\). Then
\[
\dim(A)=\frac{A_1!}{\prod_i(A_i-A_{i+1})!}=\prod_i\binom{A_{i-1}}{A_i},
\]
with \(A_0:=A_1\). The multinomial dimension factors into a chain of binomial coefficients.
\end{definition}

\begin{remark}[Boundary via survival functions]
\label{rem:boundary}
The boundary point is represented by the benchmark sequence \(1>s_1>s_2>\cdots>0\). The Martin kernel takes the form
\[
K_A(s)=\Bigl(\prod_i\binom{A_{i-1}}{A_i}\Bigr)\prod_i(s_i-s_{i+1})^{A_i-A_{i+1}}.
\]
The quantities \(s_i\) are deterministic coordinates of a boundary point, obtained as limiting survival probabilities \(s_i=\lim_n\PP[X>s_i\mid\text{hierarchy at level }n]\); randomness enters only through the random benchmark process which generates those coordinates. This is the survival-coordinate representation of the classical Martin-boundary description of the Kingman graph and its refinement by compositions.
\end{remark}

\subsection{From decreasing to increasing signatures}

The natural coordinates of the benchmark model are the strictly decreasing tail sequences 
\[
W=(W_1>W_2>\cdots>W_d>0),\qquad W_i\in\mathbb{N}.
\]
(The weak form \(W_1\geq W_2\geq\cdots\) is obtained by allowing equal consecutive values, corresponding to empty occupancy intervals.) These sequences are the \emph{decreasing signatures}. The associated occupancy (difference) composition is
\[
w_i=W_i-W_{i+1}\qquad(W_{d+1}=0),
\]
so that \(w=(w_1,w_2,\dots)\) is a composition of \(W_1\) -- distinct from a benchmark rate despite the shared letter, since here \(w_i\) is a coordinate of a single fixed signature rather than an environment weight. The reverse (increasing) signature is obtained by reading the parts of \(w\) from right to left, or equivalently by conjugating the diagram.

In the theory of quasisymmetric functions \cite{stanley1972,StanleyEC2,GesselQuasi1984} the two natural bases are indexed by these signatures. The monomial basis \(M_w\) is adapted to the occupancy coordinates \(w\); the fundamental (Gessel) basis \(F_W\) is adapted to the tail coordinates \(W\):
\[
F_W=\sum_{\Gamma\succeq W}M_\Gamma,
\]
where the sum runs over all refinements. The involution that reverses a composition interchanges the two bases (up to the reverse map) and therefore interchanges the geometric meanings ``occupancy between consecutive benchmarks'' and ``number of future points that exceed a given benchmark level''.

The branching graph of decreasing signatures has vertices the finite strictly decreasing sequences of positive integers, graded by the top entry \(n=W_1\). An edge \(W\to W'\) exists when \(W'\) is obtained from \(W\) by increasing one entry by \(1\) while preserving the strict inequalities (or by adjoining a new final entry equal to \(1\)). A few levels are shown below (labels written in compact form: \(21\) means \((2,1)\), \(32\) means \((3,2)\), etc.).

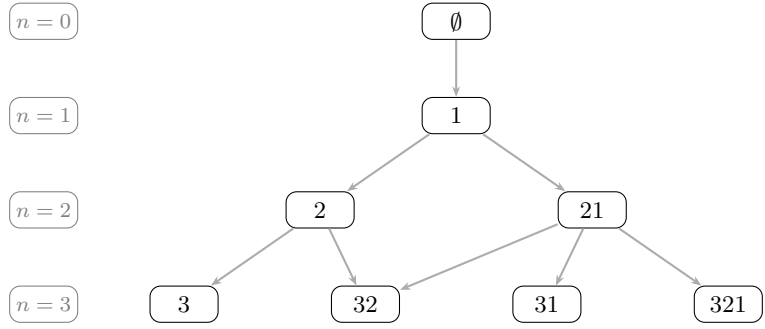
\begin{center}
\begin{tikzpicture}[
  every node/.style={font=\footnotesize, draw, rounded corners, minimum width=0.9cm, minimum height=0.48cm, inner sep=1.2pt, align=center},
  arr/.style={-{Stealth[length=1.4mm]}, thick, gray!65},
  level/.style={font=\scriptsize\itshape, gray}
  ]
  \node (empty) at (0,0) {$\emptyset$};
  \node (1) at (0,-1.25) {$1$};
  \node (2) at (-1.8,-2.5) {$2$};
  \node (21) at (1.8,-2.5) {$21$};
  \node (3) at (-3.6,-3.75) {$3$};
  \node (32) at (-1.2,-3.75) {$32$};
  \node (31) at (1.2,-3.75) {$31$};
  \node (321) at (3.6,-3.75) {$321$};
  \draw[arr] (empty) -- (1);
  \draw[arr] (1) -- (2);
  \draw[arr] (1) -- (21);
  \draw[arr] (2) -- (3);
  \draw[arr] (2) -- (32);
  \draw[arr] (21) -- (32);
  \draw[arr] (21) -- (31);
  \draw[arr] (21) -- (321);
  \node[level, left] at (-5.0,0) {$n=0$};
  \node[level, left] at (-5.0,-1.25) {$n=1$};
  \node[level, left] at (-5.0,-2.5) {$n=2$};
  \node[level, left] at (-5.0,-3.75) {$n=3$};
\end{tikzpicture}
\end{center}


\subsection{Sampling formula}

The general composition law of Theorem~\ref{thm:general} specialises, under the decreasing-signature coordinates, to an explicit sampling formula for the occupancy vector \(\mu\) (or equivalently for the tail sequence \(W\)). When the benchmarks are the ordered points of a unit-rate Poisson process (matched-rate case) the composition is uniform on all compositions of \(n\), which in \(W\)-coordinates becomes

\[
\PP[W]=\frac1{\binom{n+k}{k}}\qquad\text{for every }W\text{ with }W_1=n\text{ and  }k\text{ parts}.
\]
In the general weighted case the probability of a concrete path that realises a given \(W\) is the product of the successive Beta integrals appearing in~\eqref{eq:general}. Sequentially, the predictive probability that the next future point falls into the cell governed by residual weight \(W\) is proportional to the normalised residual weight:
\[
\PP[\text{next point in cell }r\mid\text{current residuals}]\ \propto\ W_r.
\]
This is the direct sampling rule induced by the size-biased residual spacings of Theorem~\ref{thm:spacings}.

\section{Nonhomogeneous stick-breaking and the Bernoulli sieve}
\label{sec:ewens}

The Ewens sampling formula itself is classical and is not re-derived at length here; what is worth deriving carefully is the more general mechanism it sits inside, since equal slot widths are only the simplest choice available.

Partition the past into consecutive slots of widths \(\theta_1,\theta_2,\dots>0\), not necessarily equal,
\[
w_j=\Bigl[-\textstyle\sum_{i\le j}\theta_i,\,-\sum_{i<j}\theta_i\Bigr),\qquad j=1,2,3,\dots,
\]
and let \(\beta_j\) be the minimum strength in slot \(w_j\). Since disjoint regions of a PPP are independent, \(\beta_1,\beta_2,\dots\) are independent with \(\beta_j\stackrel{\rm d}{=}\Exp(\theta_j)\) -- identical only if the widths are.

\begin{proposition}[Nonhomogeneous stick-breaking]
\label{prop:gem}
Let \(X\stackrel{\rm d}{=}\Exp(1)\) be a single future point, independent of the hierarchy, and let \(P_j=\PP[X\in C_j\mid\beta_1,\beta_2,\dots]\), \(C_j=(S_{j-1},S_j)\), \(S_j=\beta_1+\cdots+\beta_j\). Then \(V_j:=e^{-\beta_j}\stackrel{\rm d}{=}{\rm B}(\theta_j,1)\) independently across \(j\), and \(P_j=(1-V_j)\prod_{l<j}V_l\). At \(\theta_j\equiv\theta\) this is the defining stick-breaking representation of \(\GEM(\theta)\); for general \((\theta_j)\) it is an independent, not necessarily homogeneous, stick-breaking sequence.
\end{proposition}

\begin{proof}
As in Section~\ref{sec:race}, \(\beta_j\sim\Exp(\theta_j)\) gives \(V_j=e^{-\beta_j}\sim{\rm B}(\theta_j,1)\), and \(P_j=e^{-S_{j-1}}-e^{-S_j}=(1-V_j)\prod_{l<j}V_l\) by telescoping. Independence of the \(\beta_j\) across disjoint slots gives independence of the \(V_j\).
\end{proof}

\begin{theorem}[Ewens sampling formula; Ewens 1972, Kingman 1975]
\label{thm:esf}
At \(\theta_j\equiv\theta\), with \(X_1,\dots,X_n\) a future sample of size \(n\) and \(A_j=\#\{i:X_i\in C_j\}\), the exchangeable partition induced by ``same cell'' has distribution
\begin{equation}
\label{eq:esf}
\PP[a_1,\dots,a_n]=\frac{n!}{\theta^{(n)}}\prod_{j=1}^n\Bigl(\frac\theta j\Bigr)^{a_j}\frac1{a_j!},
\qquad\theta^{(n)}=\theta(\theta+1)\cdots(\theta+n-1).
\end{equation}
\end{theorem}

The genuinely general case -- widths \((\theta_j)\) varying arbitrarily rather than constant or growing linearly as in Section~\ref{sec:constrained} -- produces an infinite-dimensional family of independent stick-breaking hierarchies, of the kind studied under the name of the \emph{Bernoulli sieve} \cite{GnedinPartitionsConstraints,GnedinIksanovMarynych2010,gnedin2010bernoulli} and, in its full record-dependent generality, in \cite{GnedinBiasRecord}. As one explicit instance, geometrically decaying widths \(\theta_j=\theta r^{j-1}\), \(r\in(0,1)\), give \(V_j\sim{\rm B}(\theta r^{j-1},1)\) independently, and
\[
\PP[X\in C_1]=1-e^{-\theta},\qquad \PP[X\in C_j\mid X\notin C_1,\dots,C_{j-1}]=1-e^{-\theta r^{j-1}}\to0\text{ as }j\to\infty,
\]
a hierarchy whose cells shrink geometrically in capture probability rather than staying stochastically constant as GEM's do. Whether the general \((\theta_j)\) construction reaches every law studied under the Bernoulli-sieve name, or only the sub-class with Beta-distributed survival factors at each step, is left open here.

\begin{remark}[Consistency checks]
\label{rem:esf-check}
\(P_1=1-e^{-\beta_1}\stackrel{\rm d}{=}{\rm B}(1,\theta)\), so \(\PP[A_1=a]=\binom na\,\theta\,{\rm B}(a+1,\theta+n-a)\) -- exactly the single-benchmark law of Section~\ref{sec:one-benchmark} applied to \(\beta_1\stackrel{\rm d}{=}\Exp(\theta)\). At \(\theta=1\) the slots \(w_j\) are the original unit slots of Section~\ref{sec:one-benchmark}, and \(\PP[A_1=a]=1/(n+1)\) for every \(a=0,\dots,n\) recovers \eqref{eq:uniform}.
\end{remark}

\section{The finite exponential race and Kingman's coalescent}
\label{sec:coalescent}

Kingman's coalescent \cite{Kingman1982} is the continuous-time Markov chain on partitions of \([n]\), starting from the partition into \(n\) singletons, in which every pair of currently existing blocks merges at rate \(1\), independently of the sizes of the blocks. When there are \(b\) blocks, the waiting time to the next merger is \(\Exp\binom{b}{2}\), and the process runs until a single block remains.

We record here the exact connection between this genealogical process and the finite, value-ranked benchmark hierarchy of Sections~\ref{sec:matched}--\ref{sec:race}, and we are careful to distinguish it from a different, and non-equivalent, construction one might be tempted to write down.

\begin{theorem}[Coalescent block sizes via matched-rate benchmarks]
\label{prop:coalescent-blocks}
Let \(B_1,\dots,B_k\) be i.i.d.\ \(\Exp(1)\) benchmarks and let the future sample of size \(n\) be generated by \({\rm S}\). Conditionally on every one of the \(k+1\) cells being nonempty, the sizes \((A_0,\dots,A_k)\), regarded as an unordered multiset attached to specific future individuals, have  the same distribution as the block sizes of Kingman's \(n\)-coalescent at the moment it has \(b=k+1\) blocks.
\end{theorem}

\begin{proof}
Write \(b=k+1\). It suffices to compare, for every specific labelled partition of \(\{1,\ldots,n\}\) into blocks of sizes \(n_1,\ldots,n_b>0\) (summing to \(n\)), the probability each model assigns to it.

\emph{Benchmark model.} Conditioned on nonemptiness, \((A_0,\ldots,A_k)\) is uniform over the \(\binom{n-1}{b-1}\)  compositions of \(n\) with \(b\) positive parts, and given the composition, the assignment of specific future individuals to specific cells is uniform over the \(n!/\prod n_i!\) possibilities. A specific labelled \emph{set} partition corresponds to \(b!\) such outcomes (one for each assignment of the \(b\) specific blocks to the \(b\) ordered cells), each of probability \(\bigl[1/\binom{n-1}{b-1}\bigr]\cdot\bigl[\prod n_i!/n!\bigr]\). Hence
\begin{equation}
\label{eq:cbench}
\PP_{\rm bench}(\text{specific partition})=c_b^{\rm bench}\prod_i n_i!,\qquad
c_b^{\rm bench}=\frac{b!}{n!\binom{n-1}{b-1}}.
\end{equation}

\emph{Coalescent.} We show by downward induction on \(b\), from \(b=n\) to \(b=1\), that 
$$\PP_{\rm coal}(\text{specific partition})=c_b^{\rm coal}\prod_i n_i!$$
for a constant \(c_b^{\rm coal}\) depending only on \((n,b)\), not on the partition. At \(b=n\), the unique partition into singletons has probability \(1=c_n^{\rm coal}\), so \(c_n^{\rm coal}=1\). For the inductive step, fix a target partition at \(b-1\) blocks of sizes \(n_1,\ldots,n_{b-1}\). Since the coalescent merges a uniformly random \emph{pair} of the \(b\) current blocks, independently of block sizes, the target is reached exactly when, at the previous step, some target block \(i\) was instead present as two separate parts \((A,B)\) with \(A\sqcup B\) equal to that block, and this specific pair merged. An elementary count gives, summed over the unordered nonempty splits of an \(m\)-set,
\[
\sum_{A\subsetneq[m],\,A\ne\emptyset}\tfrac12\,|A|!\,(m-|A|)!=\tfrac12\sum_{j=1}^{m-1}\binom mj j!(m-j)!=\tfrac12(m-1)\,m!.
\]
Summing over which of the \(b-1\) target blocks was the freshly merged one,
\[
\PP_{\rm coal}(\text{target})=\frac1{\binom b2}\sum_{i=1}^{b-1}c_b^{\rm coal}\Bigl(\prod_{j\ne i}n_j!\Bigr)\cdot\tfrac12(n_i-1)\,n_i!
=c_b^{\rm coal}\Bigl(\prod_jn_j!\Bigr)\cdot\frac{\sum_i(n_i-1)}{2\binom b2},
\]
and since \(\sum_{i=1}^{b-1}(n_i-1)=n-(b-1)\), this gives a recursion independent of the target partition:
\begin{equation}
\label{eq:crecur}
c_{b-1}^{\rm coal}=c_b^{\rm coal}\cdot\frac{n-b+1}{b(b-1)},\qquad c_n^{\rm coal}=1.
\end{equation}

\emph{Matching.} The constant \(c_b^{\rm bench}\) of \eqref{eq:cbench} satisfies \(c_n^{\rm bench}=1\) (immediate at \(b=n\)) and, using  \(\binom{n-1}{b-1}=\binom{n-1}{b-2}\cdot\frac{n-b+1}{b-1}\),
\[
c_b^{\rm bench}\cdot\frac{n-b+1}{b(b-1)}
=\frac{(b-2)!}{n!}\cdot\frac{n-b+1}{\binom{n-1}{b-1}}
=\frac{(b-2)!}{n!}\cdot\frac{b-1}{\binom{n-1}{b-2}}
=\frac{(b-1)!}{n!\binom{n-1}{b-2}}=c_{b-1}^{\rm bench},
\]
giving recursion~\eqref{eq:crecur}. Since \(c^{\rm coal}\) and \(c^{\rm bench}\) satisfy the same recursion from the same boundary value at \(b=n\), they coincide for every \(1\le b\le n\); the two laws on specific labelled partitions -- hence on unordered block sizes -- are identical.
\end{proof}


The correspondence is a genuine identity between two combinatorially different objects: a benchmark cell may be empty (no future point at all falls below the first benchmark), whereas a coalescent block is by construction never empty. Conditioning on nonemptiness is therefore not a technicality to be dropped; it is exactly what reconciles the two pictures.


\section{Sampling versus insertion}
\label{sec:insertion}

Sampling classifies a future point against a fixed hierarchy: \(A\mapsto A+e_r\). Insertion allows a future point to become a new benchmark, altering the hierarchy itself. The distinction is not as sharp as the data alone can show: the factorial code of a permutation represents it as a sequence of ranks-among-those-seen-so-far, and this Lehmer-code sequence is generated identically whether the \(i\)-th future point is ranked against a fixed external hierarchy that happens to have revealed exactly \(i\) cells by that time, or against the growing set of the first \(i-1\) future points themselves, now acting as their own benchmarks. Sampling and insertion are different mechanisms, with different consistency properties, but the same rank data is compatible with either story -- an identifiability paradox that the rest of this section resolves by treating the two as distinct constructions rather than attempting to tell them apart after the fact.

Sampling remains fundamental despite this: Section~\ref{sec:ewens} shows that sampling alone, against a suitably infinite and slot-ordered hierarchy, already produces the Ewens sampling formula, with no reinforcement of the past by the future required. Section~\ref{sec:record-model} realises insertion concretely: each future point tests itself against the current extremes of everything seen so far and, with the appropriate probability, becomes a new benchmark in its own right. This yields the two-parameter \((\theta,\zeta)\) record family, sufficient for the joint count of lower and upper records -- a genuine second parameter, but not the Ewens--Pitman discount; that one is reached differently, in Section~\ref{sec:constrained}.

Predictive insertion probabilities, when insertion is present, follow at once from the composition law by Bayes:
\[
\PP[A\to A+e_r\mid A]\propto\frac{\PP[A+e_r]}{\PP[A]}.
\]
No separate construction is required.

A choice operator need not retain a single point per slot: it may instead run a tournament among several points and retain the outcome of a comparison rule with more than one round. Fix an integer \(\nu\), pair up \(2\nu\) i.i.d.\ points, retain the pairwise minima, and then retain the maximum of those minima. The resulting residual weight follows the Topp--Leone law, density \(2\nu\,p^{\nu-1}(1-p)(2-p)^{\nu-1}\) on \((0,1)\), with rational binomial moments and a closed-form predictive rule for every integer \(\nu\); at \(\nu=1\) it is the ordinary matched-rate case of Section~\ref{sec:matched}. This is a genuinely non-monotone choice operator -- the first round favours small values, the second favours the larger of the survivors -- unlike every coordinatewise-minimum construction used elsewhere in this paper, and it is the source of the tournament structure taken up again in Section~\ref{sec:power-product}.

Insertion mechanisms richer than Section~\ref{sec:record-model}'s -- an arbitrary sub-process of the past quadrant used as a splitter, or a choice operator producing fixed points -- are left to a companion paper; a full treatment needs to weigh the exchangeable-selection framework of \cite{GnedinKrengel1996} against the present geometry, which is more than a single section can honestly carry.

\section{Insertion realised: the Foster--Stuart record model}
\label{sec:record-model}

We now realise the insertion mechanism promised in Section~\ref{sec:insertion}, following \cite{GnedinRecords2007,GnedinBiasedRecords2011}
Let \(X_1,X_2,\dots\) be the future population sampled by \({\rm S}\), regarded as a growing sequence rather than a fixed batch of size \(n\). Instead of testing \(X_n\) against a separately drawn past \(\Pi_-\), let the benchmarks be generated by the sequence itself: \(X_j\) is a \emph{lower record} if \(X_j<\min(X_1,\dots,X_{j-1})\), an \emph{upper record} if \(X_j>\max(X_1,\dots,X_{j-1})\), and \(X_1\) is the (improper) \emph{centre}. Write \(\ell,u\) for the numbers of proper lower and upper records among \(X_1,\dots,X_n\), and \(\pi_n\) for the ranking permutation of \((X_1,\dots,X_n)\). The pair \((\ell,u)\) is exactly the statistic underlying the classical Foster--Stuart test for trend \cite{FosterStuart1954}, a linear function of the numbers of upper and lower records in a time series; the model below is its natural two-parameter exchangeable generalisation.

\begin{proposition}[Two-parameter sufficiency, {\cite[Prop.~2]{GnedinRecords2007}}]
\label{prop:record-family}
For \(\theta,\zeta>0\),
\[
\PP^{(\theta,\zeta)}_n(\pi_n)=\frac{\theta^\ell\zeta^u}{(\theta+\zeta)_{n-1}},\qquad
(\theta+\zeta)_{n-1}:=(\theta+\zeta)(\theta+\zeta+1)\cdots(\theta+\zeta+n-2),
\]
defines a law on rankings of \(n\) points under which \(\pi_n\) is uniform conditionally on \((\ell,u)\); at \(\theta=\zeta=1\) it is the ranking law of an i.i.d.\ sample, as throughout Sections~\ref{sec:one-benchmark}--\ref{sec:race}.
\end{proposition}

\begin{proposition}[Insertion rule, {\cite[eq.~(5)]{GnedinRecords2007}}]
\label{prop:insertion-rule}
Under \(\PP^{(\theta,\zeta)}\), the point \(X_{n+1}\), tested against the current extremes, becomes a new lower record with probability \(\theta/(\theta+\zeta+n-1)\), a new upper record with probability \(\zeta/(\theta+\zeta+n-1)\), and otherwise falls into one of the \(n-1\) already-resolved interior ranks, each with probability \(1/(\theta+\zeta+n-1)\).
\end{proposition}

This is a concrete instance of the Bayes identity of Section~\ref{sec:insertion}: the two extreme outcomes are exactly the moves that extend the benchmark hierarchy, weighted \(\theta\) and \(\zeta\) against a unit weight for every already-open interior cell.

\begin{theorem}[Two-sided GEM limit, {\cite[Prop.~7]{GnedinRecords2007}}]
\label{thm:two-sided-gem}
As \(n\to\infty\) under \(\PP^{(\theta,\zeta)}\), the scaled record values converge a.s.\ to a limit \((\rho_k)_{k\in\ZZ}\), with \(\rho_0\stackrel{\rm d}{=}{\rm B}(\theta,\zeta)\); conditionally on \(\rho_0\), the sequences \((\rho_k)_{k<0}\) and \((\rho_k)_{k>0}\) are independent, and
\[
\rho_k=\rho_0T_kT_{k+1}\cdots T_{-1}\ (k<0),\qquad
\rho_k=1-(1-\rho_0)Z_1Z_2\cdots Z_k\ (k>0),
\]
with \(T_j\stackrel{\text{i.i.d.}}{\sim}{\rm B}(\theta,1)\), \(Z_j\stackrel{\text{i.i.d.}}{\sim}{\rm B}(\zeta,1)\), all independent of \(\rho_0\).
\end{theorem}

\begin{remark}
Conditionally on \(\rho_0\), the lower side alone reproduces exactly the \(\GEM(\theta)\) stick-breaking of Section~\ref{sec:ewens} on \([0,\rho_0]\), and the upper side the \(\GEM(\zeta)\) stick-breaking on \([\rho_0,1]\). Setting \(\zeta=1\) specialises \(\PP^{(\theta,1)}\) to Ewens' distribution, matching Section~\ref{sec:ewens} at the level of the one-sided law; \(\theta=1\) gives the mirror specialisation \(\PP^{(1,\zeta)}\).
\end{remark}

\begin{theorem}[Canonicity, {\cite[Prop.~6]{GnedinRecords2007}}]
\label{thm:canonicity}
Every coherent sequence of laws on rankings of \([n]\), \(n=1,2,\dots\), under which \(\pi_n\) is uniform conditionally on \((\ell,u)\) for every \(n\), is a unique mixture of the \(\PP^{(\theta,\zeta)}\) (including the degenerate boundary cases at \(\theta=0\) or \(\zeta=0\)).
\end{theorem}

This is the sense in which \((\theta,\zeta)\) is a genuine, distinguished two-parameter family, in the same spirit as the matched-rate uniqueness of Section~\ref{sec:matched}. It should be stressed that \((\theta,\zeta)\) is \emph{not} the EPR pair \((\alpha,\theta)\): it governs the two-sided \emph{count} of extremes, not a discount reshaping block sizes. 

\begin{remark}[Record-dependent measures]
Theorem~\ref{thm:canonicity} conditions only on the pair of counts \((\ell,u)\). A finer sufficiency statistic -- the actual set of upper-record positions rather than just their number -- is analysed in \cite{GnedinBiasRecord}, where the extreme record-dependent measures on \(\mathfrak S_n\) are classified for every \(n\) coherently; this is the natural common refinement of Theorem~\ref{thm:canonicity} above and of the central-measure classification for the factorial (insertion) graph taken up in Section~\ref{sec:vershik-lodkin}.
\end{remark}

\section{\(r\)-records and constrained partitions} 

\label{sec:constrained}

The choice operator of Section~\ref{sec:one-benchmark} retains the \emph{minimum} of each slot.  Following \cite{GnedinConstrained2006} we show what happens when the choice operator instead waits for a prescribed number of points before committing to a benchmark -- the \(r\)-record idea.

\begin{definition}[Constrained hierarchy]
Fix positive integers \(\rho=(\rho_1,\rho_2,\dots)\). A benchmark hierarchy \(H_0=1>H_1>H_2>\cdots\) on \([0,1]\) is \(\rho\)-constrained if \(H_k\) is obtained from \(H_{k-1}\) by screening future points independently and uniformly on \([0,H_{k-1})\) until \(\rho_k\) of them have arrived, and setting \(H_k\) to the value of the founding point closest to \(H_{k-1}\) among these \(\rho_k\).
\end{definition}

In strength language: rather than taking the single minimum of a slot as its benchmark (\(\rho\equiv1\), Section~\ref{sec:one-benchmark}), the choice operator now waits for \(\rho_k\) points to fall below the previous threshold and reports the largest of them as the new benchmark -- an \(r\)-record.

\begin{proposition}[Stick-breaking with varying parameters, {\cite[eqs.~(6)--(8)]{GnedinConstrained2006}}]
\label{prop:constrained-stick}
If the residual fractions \(W_k=H_k/H_{k-1}\) are independent with \(W_k\stackrel{\rm d}{=}{\rm B}(a_k,b_k)\), the induced composition of a future sample has an explicit product-form law built from Polya--Eggenberger decrement matrices determined by \((a_k,b_k)_{k\ge1}\), generalising the constant-parameter law of Section~\ref{sec:ewens}, which is the case \(\rho\equiv1\), \(a_k\equiv\theta\), \(b_k\equiv1\).
\end{proposition}

\begin{theorem}[Ewens--Pitman--Yor embedding, {\cite[\S6]{GnedinConstrained2006}}]
\label{thm:ewens-pitman}
Take \(\rho\equiv1\) and let the benchmark weights grow linearly in \(k\): \(a_k=\theta+k\alpha\), \(b_k=1-\alpha\), for \(\alpha\in[0,1)\), \(\theta>-\alpha\). The resulting composition is the two-parameter Ewens--Pitman--Yor partition with discount \(\alpha\) and concentration \(\theta\) -- the same family whose power-law tail makes it, among other things, the standard prior for word-frequency distributions in hierarchical Bayesian language modelling \cite{teh2006hierarchical}.
\end{theorem}

\begin{remark}
This identifies the missing ingredient for reaching the full Ewens--Pitman--Yor family: not insertion, and not a varying \(\rho\), but benchmarks of \emph{linearly growing} weight \(a_k=\theta+k\alpha\) along a single slot-ordered hierarchy, in place of the constant weight \(\theta\) of Section~\ref{sec:ewens}. Setting \(\alpha=0\) gives \(a_k\equiv\theta\), \(b_k\equiv1\), the i.i.d.\ \({\rm B}(\theta,1)\) construction of Proposition~\ref{prop:gem}: the Ewens sampling formula of Section~\ref{sec:ewens} is the \(\alpha=0\) member of this larger family. The second parameter \(\alpha\) is a rate of growth of benchmark weight along the hierarchy, unrelated to the \((\theta,\zeta)\) pair of Section~\ref{sec:record-model}, which instead doubles a constant-weight hierarchy on two sides rather than letting one side grow.
\end{remark}

\begin{theorem}[Number of blocks, {\cite[Prop.~8]{GnedinConstrained2006}}]
Under the i.i.d.\ stick-breaking \(H_k=W_1\cdots W_k\) with \(\mu=\EE[-\log W_1]\), \(\sigma^2=\operatorname{Var}[-\log W_1]\) finite and \(\log(\sum_{j\le k}\rho_j)=o(k)\), the number of occupied cells \(K_n\) among a future sample of size \(n\) satisfies \(K_n\sim\mu^{-1}\log n\) a.s., and \((K_n-\EE K_n)/\sqrt{\operatorname{Var}K_n}\) converges to a standard Gaussian, with \(\EE K_n\sim\mu^{-1}\log n\), \(\operatorname{Var}K_n\sim\sigma^2\mu^{-3}\log n\).
\end{theorem}

\begin{proof}[Proof idea]
\(K_n\) is squeezed between the number of renewal epochs of the walk \(-\log W_k\) below \(\log n\) and a perturbation controlled by the future points falling near the current threshold; both are asymptotically Gaussian by the renewal CLT, and a large-deviation argument closes the gap. See \cite{GnedinConstrained2006} for the full argument.
\end{proof}

\section{Regenerative hierarchies and conjugate updating}
\label{sec:regenerative}

We now answer directly which benchmark weights admit simple, conjugate predictive updates.

\subsection{Fixed weights: only equality is conjugate}

\begin{proposition}
\label{prop:only-equal}
In the finite exponential race of Section~\ref{sec:race}, with a fixed multiset of weights \(w_1,\ldots,w_k\) discovered in a random order, the residual weight \(W_2(\sigma)=\sum_iw_i-w_{i_1}\) is independent of \(\sigma\) -- and hence, by the same argument at every level, the general composition law~\eqref{eq:general} collapses from a sum of \(k!\) terms to a single term -- if and only if \(w_1=\cdots=w_k\).
\end{proposition}

\begin{proof}
If \(W_2(\sigma)\) does not depend on \(\sigma\), then \(w_{i_1}=\bigl(\sum_iw_i\bigr)-W_2(\sigma)\) does not depend on which benchmark is discovered first, i.e.\ \(w_1=\cdots=w_k\). Conversely, with all weights equal, \(W_r(\sigma)=(k-r+1)w\) for every \(\sigma\), as already noted after Theorem~\ref{thm:general}.
\end{proof}

So within a fixed, finite collection of identity-bearing benchmarks, matched rates (Proposition~\ref{prop:matched}) are not a convenient special case: they are the \emph{only} case admitting a one-term, conjugate composition law. Richer conjugate families must come from changing the architecture of the hierarchy, not from retuning weights within a fixed finite race.

\subsection{Positional hierarchies: regenerative composition structures}

In Sections~\ref{sec:ewens} and~\ref{sec:constrained} the weight of the \(r\)-th benchmark is assigned to \emph{position} \(r\) directly (constant \(\theta\), or linearly growing \(\theta+r\alpha\)), so there is no discovery order to sum over. In \cite{GnedinPitman2005} exactly the composition laws with this property are classified.

\begin{definition}[Regenerative hierarchy]
\label{def:regenerative}
A benchmark hierarchy is regenerative if, conditionally on the size \(m\) of the first cell, the remaining hierarchy and future points form an independent copy of the same model on a population of size \(n-m\).
\end{definition}

\begin{theorem}[{\cite[Prop.~3.1, Thm.~5.2]{GnedinPitman2005}}]
\label{thm:regenerative}
A composition law is regenerative if and only if it has the product form
\[
p(n_1,\ldots,n_k)=\prod_{j=1}^k q(N_j:n_j),\qquad N_j=n_j+\cdots+n_k,\qquad q(n:m):=\frac{\Phi(n:m)}{\Phi(n)},
\]
for a nonnegative array \(\Phi(n:m)\) arising from the Laplace exponent \(\Phi\) of a subordinator; equivalently, if and only if the benchmarks arise as a standard exponential sample separated by the closed range of a subordinator with that Laplace exponent. (Here \(\Phi\) is unrelated to the comparison-count coordinates \(W\) of Section~\ref{sec:algebra}.)
\end{theorem}

\begin{remark}[Our two infinite hierarchies are the regenerative examples]
Section~\ref{sec:ewens}'s construction is exactly the \((0,\theta)\) case of \cite{GnedinPitman2005}, \S8.2: the ordered Ewens formula from \(\operatorname{beta}(1,\theta)\) stick-breaking. Section~\ref{sec:constrained}'s construction is exactly the general \((\alpha,\theta)\) case, \S8, with \(\Phi(s)=s{\rm B}(1-\alpha,s+\theta)\). By \cite[Thm.~8.1]{GnedinPitman2005}, \((\alpha,\theta)\) is -- with only degenerate exceptions -- the \emph{unique} two-parameter regenerative family. Our two infinite hierarchies are therefore the one- and two-parameter extremes of the entire class of benchmark hierarchies admitting conjugate updating, not ad hoc choices.
\end{remark}

\begin{remark}[Answer to the motivating question]
A collection of benchmark weights admits a nice conjugate form with easy updates precisely when either (a) it is a finite collection of \emph{equal} weights (Proposition~\ref{prop:matched}), or (b) it is an infinite, positionally indexed sequence arising as decrements of a subordinator's Laplace exponent -- of which the \((\alpha,\theta)\) family exhausts the two-parameter cases. The general finite race of Section~\ref{sec:race} has no richer conjugate sub-family hiding inside it: its \(k!\)-term sum is the exact combinatorial price of fixing benchmark \emph{identities} rather than benchmark \emph{positions}.
\end{remark}

\section{The nonlinear seating chain \(W(y)=c\,y^\gamma\): a genuine departure from the Beta world}
\label{sec:three-param}

The two-parameter families obtained so far remain inside the gamma--beta algebra. Racing raw \(\Exp(1)\) sample points directly against an inhomogeneous Poisson process with intensity \(cy^\gamma\) on the strength axis does \emph{not} admit a closed product formula for \(\gamma\neq0\): the discovery-sequence probability is a Poisson-gap integral whose integrand mixes a linear exponential (the sample) with a degree-\((\gamma+1)\) exponential (the benchmark mean measure), and these recombine into a single tractable exponent only at \(\gamma=0\). What survives is the seating rule itself, defined directly as a Markov chain rather than derived from an exchangeable partition.

\subsection{Model} Consider the composition-valued Markov chain (on strict compositions, no zero parts) \(\Lambda_n=(n_1,\dots,n_k)\), \(\sum_in_i=n\), with transition probabilities
\begin{eqnarray*}
\PP\bigl[(n_1,\dots,n_j,\dots,n_k)\to(n_1,\dots,n_j+1,\dots,n_k)\bigr]&=&\frac{n_j^{1/\gamma}}{S(\lambda)+c},\qquad  \\
\PP\bigl[(n_1,\dots,n_k)\to(n_1,\dots,n_k,1)\bigr]&=&\frac c{S(\lambda)+c},
\end{eqnarray*}
where
\begin{equation}
\label{eq:nlCRP}
S(\lambda)=\sum_in_i^{1/\gamma}
\end{equation}
replaces the total mass \(n\) that drives the ordinary CRP. At \(\gamma=1\) this is exactly the seating rule of Section~\ref{sec:ewens}.

\subsubsection{What it is not} This is not the CRP, not Pitman--Yor, and not a Gibbs partition \cite{gnedin2006gibbs}: a Gibbs partition's transition probabilities depend on the composition only through \((n,k)\) and the size of the block being joined, but here the normalising denominator \(S(\lambda)+c\) depends on the \emph{whole} shape of \(\lambda\) through the nonlinear sum \(\sum_in_i^{1/\gamma}\). Consequently there is no homogeneous decrement matrix \(q_n(m)=\PP[\text{first part has size }m\mid n]\) of the regenerative type: the chain is well-defined at every step -- it is an honest Markov chain on compositions, with no consistency issue, since nothing claims it is the restriction of a fixed infinite object -- but it is not Gibbs, not regenerative, and not (for \(\gamma\neq1\)) reducible to a size-\(n\)-only sampling formula.

\subsubsection{Self-similarity of the underlying intensity.} The Lévy measure of the pure power is \(\rho({\rm d}x)=c\,x^{-1-\gamma}\,{\rm d}x\). Restricting to \((0,x)\) and rescaling \(y=xz\) gives
\[
\rho_x({\rm d}y)=c\,x^{-\gamma}\,z^{-1-\gamma}\,{\rm d}z\qquad(z=y/x\in(0,1)),
\]
so the shape (the exponent \(-1-\gamma\)) is preserved and only the scale parameter flows, \(c\mapsto cx^{-\gamma}\): the construction is quasi-invariant under restriction and rescaling, with the multiplicative cocycle \(x^{-\gamma}\) taking the place of the identity that strict self-similarity would require \cite{GnedinOlshanskiQuasiInvariance2010}.

\subsubsection{The special case \(\gamma=1\).} Here \(S(\lambda)=\sum_in_i=n\) identically, so the chain reduces to \(\PP[\text{join }j]=n_j/(n+c)\), \(\PP[\text{new}]=c/(n+c)\) -- the ordinary CRP, yielding the Ewens partition structure of Section~\ref{sec:ewens}. The scale flow degenerates, \(c\mapsto c\), since \(x^{-\gamma}=x^{-1}\) no longer enters the relevant statistic: the cocycle is trivial, self-similarity holds outright rather than up to a correction, and \(\gamma=1\) is the unique value at which the chain is Gibbs.

\subsubsection{The case \(\gamma\neq1\)} Deleting the first block leaves a composition from the same family but with a transformed scale \(c'=c\,X^{-\gamma}\), \(X\) the random scaling factor associated with the first cut, so decrement probabilities \(q_n(m\mid c)\) are expected to satisfy an integral recursion of the schematic form
\[
q_n(m\mid c)=\int K_c({\rm d}x)\,q_{n-m}(\cdot\mid cx^{-\gamma})
\]
rather than a stationary decrement matrix -- a genuinely different, non-regenerative kind of consistency, not yet worked out in closed form.

\subsubsection{Interpretation.} The process is best viewed as: a Markov chain on compositions, well-defined at every \(n\) with no consistency question to resolve because none is claimed; a nonhomogeneous composition structure in the sense of \cite{PitmanTran2015}; a quasi-invariant deformation of the Ewens chain, self-similar up to the cocycle \(x^{-\gamma}\), with the cocycle trivial and the Gibbs property recovered only at \(\gamma=1\). The statistic \(S(\lambda)=\sum_in_i^{1/\gamma}\) of~\eqref{eq:nlCRP} replaces the total mass \(n\) that drives the ordinary Ewens case, and is the reason this family sits outside every closed-form class classified elsewhere in this paper.

\section{The \((\gamma,\alpha,\theta)\) hierarchical splitting}
\label{sec:power-product}

The tournament choice operator of Section~\ref{sec:insertion} shows that a multi-round comparison rule among past points can produce a residual weight with a density on \((0,1)\) that is not simply Beta. Tournament structures of this kind, varied over the number and shape of their rounds, generate a wide range of such densities; the pure-power seating chain of Section~\ref{sec:three-param} is one instance, and the following is another, chosen for the closed form it admits.

Let \(W_1,W_2,\dots\) be i.i.d.\ with density \({\rm B}(\gamma,\theta)\) on \((0,1)\), and build the stick-breaking sequence \(P_j=W_j\prod_{i<j}(1-W_i)\) as in Section~\ref{sec:ewens}. The decrement matrix
\[
q_{\gamma,\theta}(n:m)=\binom nm\,\frac{(\gamma)_m(\theta)_{n-m}}{(\gamma+\theta)_n-(\theta)_n}
\]
is a ratio of Pochhammer symbols, properly normalised at every \(n\) by construction, and reduces to the Ewens rising factorial at \(\gamma=1\). It is not the raw law of the count below \(W_1\) alone: a first break can land with none of the \(n\) points below it, and such an empty cell is not a part of the composition, which by convention lists only its nonzero blocks. The subtracted term \((\theta)_n\) is exactly this empty-cell mass -- the \(m=0\) term of the Vandermonde--Chu identity \(\sum_{m=0}^n\binom nm(\gamma)_m(\theta)_{n-m}=(\gamma+\theta)_n\) -- so \(q_{\gamma,\theta}(n:m)\) is the law of the first break's count conditioned on landing nonempty, matching the size of the first \emph{recorded} part however many silent empty breaks preceded it.

This combines with the \((\alpha,0)\) self-similar (\(\alpha\)-stable) mechanism by sliced splitting \cite{GnedinIksanovMarynych2010,gnedin2005markov}: split \(n\) first by the outer \((\gamma,\theta)\) stick-breaking, then fragment each resulting block independently by the inner \((\alpha,0)\) mechanism, decrement \(h(m)=\alpha(1-\alpha)_{m-1}/m!\) for \(m\) less than the block size, \(q(k:k)=(1-\alpha)_{k-1}/(k-1)!\) for the whole block. This defines a three-parameter regenerative composition structure, in the same spirit as the self-similar continuum-tree constructions of \cite{pitman2009regenerative}. At \(\alpha=0\) the inner step is the identity, since \(h(\cdot,0)\equiv0\) forces immediate termination once renormalised, recovering the pure \((\gamma,\theta)\) family; at \(\gamma=1\) the first-part law reduces to the two-parameter decrement matrix
\[
q_{\alpha,\theta}(n:m)=\binom nm\frac{(1-\alpha)_{m-1}}{(\theta+n-m)_m}\cdot\frac{(n-m)\alpha+m\theta}n
\]
of the ordinary \(\mathrm{EPY}(\alpha,\theta)\) family. The \((\gamma,\theta)\) family and the sliced-splitting mechanism are both due to \cite{GnedinIksanovMarynych2010}; placing \(\gamma\) at the outer level of the splitting, rather than a bare Poisson skeleton, gives the three-parameter object above. Related constructions, replacing independence in the stick-breaking factors by a Markov dependence while retaining the same marginal law, are studied by \cite{PitmanTran2015}.

\subsection{The tournament family of Lévy measures}

The general tournament of Section~\ref{sec:insertion} -- group \(m\) points, retain the minimum of each group, then retain the maximum of \(\gamma\) such minima -- generates a whole family of Lévy measures on \((0,1]\) once the tournament is fed into a benchmark hierarchy rather than a single comparison. Writing \(x\) for the residual weight and allowing an additional singularity \(\alpha\le0\) at the origin and an intensity tilter \(\theta>0\),
\begin{eqnarray}\nonumber
\tilde\nu_{m,\gamma,\alpha,\theta}({\rm d}x)
=\theta\,x^{\alpha-1}\,m\gamma\,(1-x)^{m-1}\bigl(1-(1-x)^m\bigr)^{\gamma-1}\,{\rm d}x,\\
\qquad 
m\in\NN,\ \gamma>0,\ \alpha\le0,\ \theta>0.
\end{eqnarray}
At \(m=2\), \(\alpha=1\), \(\theta=1\) this coincides with the Topp--Leone density of Section~\ref{sec:insertion} with \(\gamma=\nu\). The measure is finite at \(\alpha=0\) and infinite at \(\alpha<0\), the regime of heavy-tailed benchmarks in the sense of \cite{markovich2008}; the associated binomial moments \(\Phi(n:m')\), and hence the decrement matrix \(q(n:m')=\Phi(n:m')/\Phi(n)\), are always a finite sum of Beta functions, becoming purely rational exactly when \(m\) and \(\gamma\) are both positive integers.

\subsubsection{Minima and maxima exchanged} Reflecting \(x\mapsto1-x\) about the centre of \([0,1]\) exchanges the two rounds of the tournament: minimum-within-group followed by maximum-across-groups becomes maximum-within-group followed by minimum-across-groups. Writing \(y=1-x\) for the residual weight of this reversed tournament,
\[
\tilde\nu^*_{m,\gamma,\alpha,\theta}({\rm d}y)=\theta\,(1-y)^{\alpha-1}\,m\gamma\,y^{m-1}\bigl(1-y^m\bigr)^{\gamma-1}\,{\rm d}y,
\]
obtained from \(\tilde\nu_{m,\gamma,\alpha,\theta}\) by substituting \(x=1-y\) throughout, tilter included. The singularity moves with the reflection, from the origin to \(y=1\): benchmarks piling up under repeated minimisation accumulate near \(0\), while benchmarks piling up under repeated maximisation accumulate near \(1\), and the two constructions are mirror images of one another rather than independent families.

\subsubsection{The base case \(\gamma=2\)} For general group size \(m\), setting \(\gamma=2\) collapses the density to a two-term polynomial difference:
\[
\tilde\nu_{m,2,\alpha,\theta}({\rm d}x)=2m\theta\,x^{\alpha-1}\Bigl[(1-x)^{m-1}-(1-x)^{2m-1}\Bigr]\,{\rm d}x,
\]
obtained by expanding \(1-(1-x)^m\) inside the \(\gamma=2\) power. Combinatorially this is a depth-two tournament tree: a single maximum-node of arity \(2\) at the root, with two independent minimum-nodes of arity \(m\) as its children -- the smallest genuine tournament tree beyond a single round, in the sense already used for minimax game trees.

\subsubsection{Non-integer and non-positive \(\gamma\)} The density above is defined for every real \(\gamma>0\), integer or not, by the same formula; only the rationality of the decrement matrix requires \(\gamma\in\NN\). At \(\gamma=0\) the round of maxima disappears and the champion degenerates to a point mass at the origin. Continuing to \(\gamma<0\) no longer describes a tournament with a max-round of that many groups, but the resulting object appears closely related to the stick-breaking-with-atom structure of the \((\alpha,0)\) case in Section~\ref{sec:regenerative}, whose Lévy measure carries an atom \(\delta_1({\rm d}x)\) alongside its continuous part; we record this connection without pursuing it further here.

\subsubsection{Rational cases} For \(m,\gamma\in\NN\) the density is a polynomial after the singularity factors out, and every decrement probability is a ratio of two polynomials in \(n\) and \(m'\):
\begin{center}
\begin{tabular}{@{}cccl@{}}
\toprule
\(m\) & \(\gamma\) & \(\alpha\) & Lévy density (factor \(\theta\) omitted) \\
\midrule
2 & 1 & \(0\) & \(2(1-x)/x\) \\
2 & 1 & \(\alpha\le0\) & \(2\,x^{\alpha-2}(1-x)\) \\
2 & 2 & \(0\) & \(4(1-x)(2-x)\) \\
2 & 2 & \(\alpha\le0\) & \(4\,x^{\alpha-1}(1-x)(2-x)\) \\
2 & 3 & \(\alpha\le0\) & \(6\,x^{\alpha}(1-x)(2-x)^2\) \\
2 & \(\gamma\in\NN\) & \(\alpha\le0\) & \(2\gamma\,x^{\alpha+\gamma-2}(1-x)(2-x)^{\gamma-1}\) \\
3 & 1 & \(\alpha\le0\) & \(3\,x^{\alpha-2}(1-x)^2\) \\
3 & 2 & \(\alpha\le0\) & \(6\,x^{\alpha-1}(1-x)^2(x^2-3x+3)\) \\
3 & \(\gamma\in\NN\) & \(\alpha\le0\) & \(3\gamma\,x^{\alpha-1}(1-x)^2\bigl(1-(1-x)^3\bigr)^{\gamma-1}\) \\
4 & \(\gamma\in\NN\) & \(\alpha\le0\) & \(4\gamma\,x^{\alpha-1}(1-x)^3\bigl(1-(1-x)^4\bigr)^{\gamma-1}\) \\
\bottomrule
\end{tabular}
\end{center}
and every higher integer pair \((m,\gamma)\) continues the pattern.

\subsubsection{Non-rational cases} As soon as \(\gamma\) is a positive real rather than an integer, with \(m\) fixed, the binomial moments remain closed-form but involve the hypergeometric function \({}_2F_1\), or equivalently a non-terminating series of Beta functions; the decrement matrix stays explicit but is no longer a ratio of polynomials. The full continuous Topp--Leone family, \(m=2\), \(\gamma>0\) arbitrary,
\[
\tilde\nu({\rm d}x)=\theta\cdot2\gamma\,x^{\alpha+\gamma-2}(1-x)(2-x)^{\gamma-1}\,{\rm d}x,
\]
is the first instance; \(m=3,4,\dots\) continue with
\[
\tilde\nu({\rm d}x)=\theta\,m\gamma\,x^{\alpha-1}(1-x)^{m-1}\bigl(1-(1-x)^m\bigr)^{\gamma-1}\,{\rm d}x.
\]
In every case the survival-function form
\[
\tilde\nu({\rm d}x)=\theta\,x^{\alpha-1}\,\frac{1-F_{m,\gamma}(x)}x\,{\rm d}x,\qquad F_{m,\gamma}(x)=\bigl(1-(1-x)^m\bigr)^\gamma,
\]
holds, tying the tournament family directly to the  decrement matrix of Section~\ref{sec:regenerative}.

\subsection{Reciprocal-integer shape: \(\gamma=\theta=1/k\)}
\label{sec:reciprocal}

Return to the \((\gamma,\theta)\) stick-breaking decrement matrix of Section~\ref{sec:power-product},
\[
q_{\gamma,\theta}(n:m)=\binom nm\,\frac{(\gamma)_m(\theta)_{n-m}}{(\gamma+\theta)_n-(\theta)_n},
\]
and set \(\gamma=\theta=1/k\) for an integer \(k\ge1\). Since \(\Gamma(1/k)\) is not an elementary constant for \(k\ge3\), there is no reason to expect \(q_{1/k,1/k}(n:m)\) to be rational; it is.

\begin{proposition}
\label{prop:reciprocal}
For every integer \(k\ge1\) and every \(1\le m\le n\),
\[
q_{1/k,1/k}(n:m)=\binom nm\cdot\frac{\displaystyle\prod_{i=1}^{m-1}(1+ik)\,\prod_{j=1}^{n-m-1}(1+jk)}{\displaystyle\prod_{i=0}^{n-1}(2+ik)-\prod_{i=0}^{n-1}(1+ik)}\,,
\]
with the convention that an empty product equals \(1\); in particular the formula is a ratio of integer polynomials in \(k\), hence rational at every integer \(k\).
\end{proposition}

\begin{proof}
Unrolling the Pochhammer symbol at \(\gamma=1/k\) separates a trivial leading factor from the rest:
\[
(1/k)_m=\prod_{r=0}^{m-1}\Bigl(\frac1k+r\Bigr)=\frac1{k^m}\prod_{r=0}^{m-1}(1+rk)=\frac1{k^m}\prod_{i=1}^{m-1}(1+ik),
\]
since the \(r=0\) term equals \(1\) identically. The same unrolling applies to \((1/k)_{n-m}\), \((2/k)_n\), and \((1/k)_n\) (the last two without the special \(r=0\) cancellation, since their products run over all of \(i=0,\ldots,n-1\) starting from \(2/k\) and \(1/k\) respectively). Substituting into \(q_{1/k,1/k}(n:m)\), every factor of \(k^{-1}\) is common to numerator and denominator and cancels, leaving the stated expression.
\end{proof}

The transcendence of \(\Gamma(1/k)\) never enters: the identity is proved entirely inside the ring of polynomials in \(k\), with the passage through Gamma functions serving only as motivation. At \(k=1\) every factor \((1+ik)\) becomes \((1+i)\) and the formula collapses to the uniform law \(q_{1,1}(n:m)=1/n\) of Eq.~\eqref{eq:uniform}; at \(k=2\) it reduces, after simplification, to the central-binomial form of Section~\ref{sec:power-product}. No closed form of this kind was found for general \((\gamma,\theta)\) away from this reciprocal-integer family; the mechanism appears specific to \(\gamma=\theta\).

\begin{corollary}[Predictive rule]
\label{cor:reciprocal-predictive}
Under \(\gamma=\theta=1/k\), given that the currently open block has already absorbed \(c\) of the \(n\) points tested against it, the next point extends the same block with probability
\[
\PP[\text{extend}\mid c,n]=\frac{1+ck}{2+nk},
\]
and escapes to test the next slot with the complementary probability \(\bigl(1+(n-c)k\bigr)/(2+nk)\). This is the posterior mean of a \({\rm B}(1/k,1/k)\) prior updated by \(c\) successes in \(n\) Bernoulli trials, and reduces at \(k=1\) to the classical add-one rule \((1+c)/(2+n)\).
\end{corollary}

\begin{remark}[A superficially similar rational case that is a different family]
\label{rem:gnedin-yakubovich}
Gnedin and Yakubovich \cite{gnedin2006recursive} study a genuinely different generalisation of stick-breaking: not a one-sided splitting but a true fragmentation, tree-indexed rather than line-indexed. In the Bessel-bridge instance of their construction a crumb splits into bridge--excursion--bridge, and \emph{both} bridge pieces continue splitting independently at every subsequent generation by an independent copy of the same rule, so that generation \(k\) carries \(2^k\) live crumbs; the count of solids produced this way grows polynomially in \(n\) with a Malthusian exponent \(\alpha^*\) governed by Crump--Mode--Jagers branching theory, solving a transcendental equation \(\psi(\alpha)=1\). One split is a genuinely bivariate (Dirichlet-distributed) draw, but the process across generations is supercritical branching, not any fixed-dimensional subordinator. They isolate a rational case of their own: when the Dirichlet parameters sum to an integer \(r\), \(\psi\) is a rational function and the exponent equation becomes polynomial of degree \(r\) in \(\alpha\).

The comparison to a one-sided family such as ours is meaningful only because of a nontrivial reduction they prove: collecting the solids by \emph{size-biased} pick, rather than in their raw genealogical (tree) order, collapses the branching structure exactly onto ordinary stick-breaking (their Lemma~9 and Proposition~8). Only after this reduction does a same-color probability \(p_{\rm GY}(n)\) directly comparable to our \(q_{\gamma,\theta}\) exist at all. At \(r=1\) it reads
\[
p_{\rm GY}(n)=\frac{(1-\gamma)_n}{(1+\gamma)_n-2(\gamma)_n}\,.
\]
This telescopes exactly, for every \(n\) and every \(\gamma\), to
\[
p_{\rm GY}(n)=\frac{(1-\gamma)_{n-1}}{(1+\gamma)_{n-1}}\,,
\]
via the three shift identities \((1+\gamma)_n=(\gamma)_{n+1}/\gamma\), \((1-\gamma)_n=(1-\gamma)_{n-1}(n-\gamma)\), and \((\gamma)_n=\gamma(1+\gamma)_{n-1}\) -- which is exactly the classical Ewens--Pitman same-color probability \(p_{\alpha,\theta}(n)=(1-\alpha)_{n-1}/(1+\theta)_{n-1}\) at \(\alpha=\theta=\gamma\). Their rational case, once unwound from a branching process to a line-indexed one, is therefore the ordinary Ewens--Pitman family, which grows benchmark weight with \emph{position}; our \(\gamma=\theta=1/k\) family is i.i.d.\ across position throughout, and was never tree-indexed to begin with. The two rational cases share a mechanism -- an integer relation among Dirichlet or Beta parameters collapsing a ratio of Gamma functions to something elementary -- without being instances of one another, and arise from constructions (fragmentation versus one-sided splitting) that are not directly comparable except after their size-biased reduction.
\end{remark}

\section{Chain records, cones, and general choice operators}
\label{sec:chain-records}

Every construction so far has taken the choice operator {\rm C} to retain minima
of intervals on a single strength axis. We now let {\rm C} act by
\emph{coordinatewise domination} in \(\RR^d\), or more generally by domination
in an arbitrary ordered probability space, and show that this single change of
geometry both explains the pure-power law of Section~\ref{sec:three-param} from
first principles and supplies the missing operational answer to when the
choice operator should stop trusting its current benchmark.

\subsection{Order types of the induced comparison order}
\label{sec:order-types}

Return to the single-axis race of Section~\ref{sec:race}, now allowing the
weight sequence \(w=(w_i)_{i\in I}\) on a countable ground set to be arbitrary,
not merely summable. Ranking the strengths \(X_i\stackrel{\rm d}=\Exp(w_i)\)
defines a random linear order \(i\triangleleft j\iff X_i<X_j\) on \(I\); this
is the size-biased order of \cite{GnedinSizeBiased2023}, already invoked in
Section~\ref{sec:intro}. Writing \(\mu(x,y)=\sum_i(e^{-w_ix}-e^{-w_iy})\) with
density \(\varphi(x)=\sum_iw_ie^{-w_ix}\) for \(x\) beyond the convergence
abscissa \(\beta\) of the underlying Dirichlet series, the order type of
\(\triangleleft\) is classified completely by the summability and
accumulation behaviour of \(w\):
\begin{itemize}[leftmargin=*]
\item \(w_i\to0\), summable (\(\Lambda:=\sum_iw_i<\infty\)): type \(\ZZ_{>0}\)
-- the regime of Sections~\ref{sec:one-benchmark}--\ref{sec:three-param};
\item \(w_i\to0\), non-summable, or \((w_i)\) with an accumulation point in
\((0,\infty)\): type \(\mathbb Q\), dense on all of \((0,\infty)\) in the
latter case even though the accumulation is local
;
\item \(w_i\to\infty\) with \(\beta:=\limsup_i(\log i)/w_i=0\): type
\(\ZZ_{<0}\), a left-sided accumulation of benchmarks at the origin with no
least element.
\end{itemize}
The last case is the precise sense in which a divergent weight sequence still
produces a coherent hierarchy of benchmarks: infinitely many past strengths
pile up at \(0\) without a minimum, so every finite future point is beaten by
cofinitely many of them, while only finitely many benchmarks exceed any fixed
positive threshold. It is this regime, not the summable one, that the chain
record mechanism below produces canonically.

\subsection{Chain records and the Mellin transform}
\label{sec:chain-mellin}

Let the past population be re-sampled as i.i.d.\ marks \(Y_1,Y_2,\dots\) in
the cube \([0,1]^d\) rather than on a single strength axis, and let {\rm C}
retain the \emph{chain records}: \(Y_n\) is retained if it is dominated
coordinatewise by every previously retained mark. Write \(R_1,R_2,\dots\) for
the successive chain records and \(H_k=h(R_k)\) for the Lebesgue measure of
the box they dominate.

\begin{theorem}[Multiplicative renewal, {\cite{GnedinChainRecords2007}}]
\label{thm:chain-renewal}
\((H_k)_{k\ge1}\stackrel{\rm d}{=}(W_1\cdots W_k)_{k\ge1}\) for i.i.d.\ copies
\(W_j\) of \(H_1\), with
\[
\PP[W\in{\rm d}s]=\frac{(-\log s)^{d-1}}{(d-1)!}\,{\rm d}s,\qquad s\in[0,1],
\qquad g(\lambda):=\EE[W^\lambda]=(\lambda+1)^{-d}.
\]
Given \((H_k)\), sojourns between records are conditionally independent
geometric with parameter \(H_k\) -- the discrete analogue of the memoryless
step of Theorem~\ref{thm:spacings} -- and the law of record occurrences given
\((H_k)\) is the \emph{same as the one-dimensional case, for every \(d\)}.
\end{theorem}

At \(d=1\), \(g(\lambda)=(\lambda+1)^{-1}\), and \(H_1\stackrel{\rm d}=W\sim
{\rm B}(1,1)\) is exactly \(V_1=1-e^{-\beta_1}\) of Proposition~\ref{prop:gem}
with \(\theta=1\): the single-slot benchmark race of Section~\ref{sec:ewens}
is the memoryless, deterministic-residual-rate special case of chain records.
General \(d\) replaces the deterministic residual weight \(\theta\) by a
genuinely random, history-dependent one, at the price of losing the
elementary closed form of Theorem~\ref{thm:esf} beyond \(d=1\).

\begin{theorem}[Universal logarithmic growth]
\label{thm:log-growth}
With \(\mu=\EE[-\log W]\), \(\sigma^2=\operatorname{Var}[-\log W]<\infty\) and
\(K_n:=\max\{k:H_k>1/n\}\),
\[
K_n\sim\frac{\log n}\mu\ \text{a.s.},\qquad \EE K_n\sim\frac{\log n}\mu,\qquad
\operatorname{Var}K_n\sim\frac{\sigma^2}{\mu^3}\log n,
\]
and \(K_n\) is asymptotically Gaussian.
\end{theorem}

This represents the constrained-hierarchy CLT already proved as the
``Number of blocks'' theorem of Section~\ref{sec:constrained}, specialised to
the i.i.d.\ case \(\rho\equiv1\); here \(\mu=\sigma^2=d\) exactly, since
\(-\log W\sim\operatorname{Gamma}(d,1)\). Thus \emph{any} log-type weight
family arising as a multiplicative renewal with finite log-moments obeys the
same \(\log n\) growth law, independent of the fine shape of \(W\) -- a
genuine universality that removes the need to hand-fit a decaying weight
sequence \(w_i=c(\log i)^p\) to a target growth rate for the number of active
benchmarks \(K_n\): dimension \(d\) (or, more generally, any renewal law
\(W\)) supplies it directly.

Poissonising the record process and solving for the moments
\(m_\beta(t)=\EE[B_t^\beta]\) of the current record height gives
$$m_\beta(t)=\sum_{k\ge0}\frac{(-t)^k}{k!}\prod_{j=0}^{k-1}\bigl(1-g(j+\beta)\bigr);$$
depoissonising at \(\beta=1\) yields the exact record probability
\begin{equation}
\label{eq:chain-pn}
p_n=\sum_{k=0}^{n-1}\binom{n-1}k(-1)^k\prod_{j=0}^{k-1}\bigl(1-g(j+1)\bigr).
\end{equation}
At \(d=1\), \(p_n=1/n\) (classical records, matching
\(\PP[A_1=n)=1/(n+1)\)-type single-slot computations up to reindexing); at
\(d=2\), \(p_n=1/(2n)\) for \(n>1\), a second exact rational case; for \(d>2\)
\eqref{eq:chain-pn} becomes a generalised \({}_dF_d\) hypergeometric series in
\(t\), with \(p_n\sim1/(dn)\) throughout.

\subsection{General cones and lower-homogeneous spaces}
\label{sec:cones}

Coordinatewise domination is domination by the positive orthant. Replace it
by an arbitrary punctured cone \(K\subset\RR^d\): the choice operator retains
\(Y_n\) whenever it lies in the \(k\)th layer
\(\mathcal L^{(k)}=\{Y_i:\#(K_{Y_i}\cap\{Y_1,\dots,Y_n\})=k-1\}\), \(K_x=x+K\).
The first layer with \(K\) the positive orthant is exactly the set of
Pareto-optimal points -- the south-east record example of
Section~\ref{sec:setup} is the case \(d=2\), \(K\) a quarter-plane -- and in
\cite{GnedinConicalExtremes1994} it is shown that, for sampling
distributions with asymptotically independent radial and angular parts and a
regularly varying radial tail, the layer counts \(V_n^{(k)}\) obey power laws
governed jointly by the tail index and by the opening angle of \(K\): cone
shape is a second free parameter, orthogonal to \(d\), controlling how often
domination occurs.

More generally, call an ordered probability space \((\mathcal Z,\mu,\prec)\)
\emph{lower-homogeneous} if the lower section \(L_v=\{u:u\prec v\}\), with
conditional measure \(\mu(\cdot)/\mu(L_v)\), is isomorphic to
\((\mathcal Z,\mu,\prec)\) itself. Records under \(\prec\) then undergo the
renewal of Theorem~\ref{thm:chain-renewal} verbatim, with factor
\(X\stackrel{\rm d}=\mu(L_V)\), \(V\sim\mu\) \cite{GnedinLastRecord2006}. The
cube with coordinatewise order is one instance; another is:

\begin{example}[Interval space]
Let \(\mathcal Z=\{\,]a,b[\,\subset[0,1]\}\), \(\prec\) inclusion, and
$$(\mu({\rm d}a\,{\rm d}b)=\alpha(\alpha-1)(b-a)^\alpha\,{\rm d}a\,{\rm d}b$$
for \(\alpha>1\). Then \(\PP[X\in{\rm d}x]/{\rm d}x=(\alpha-1)(x^{-1/\alpha}-1)\),
a one-parameter power-law family distinct from the \(\log\)-power families
above, arising from a domination geometry rather than by curve-fitting.
\end{example}
Every such \(X\) satisfies \(\PP[X\le x]\ge x\) for \(x\in[0,1]\), a necessary
compatibility condition inherited from lower-homogeneity that any candidate
weight family built this way must satisfy.

\subsection{The choice operator as a self-similar process}
\label{sec:stopping}

The benchmark ladder of Section~\ref{sec:chain-mellin} is the jump chain of a
continuous-time self-similar Markov process \(R=(R_t)_{t\ge0}\): from current
benchmark \(r\), jump at rate \(r\) to \(rX\), so that
\((R_t\mid R_0=r)\stackrel{\rm d}=(rR_{rt}\mid R_0=1)\). Unwound, this is once
more the exponential race, with \(r_0X_1\cdots X_{k-1}\) playing the role of
the residual weight \(W_r(\sigma)\) of Theorem~\ref{thm:spacings}. This self-similar representation is a genuine, if narrow, connection to the surrounding chain-record material; the associated theory of when to stop the process before a horizon \(T\) \cite{GnedinLastRecord2006} belongs to optimal-stopping rather than to partition or sampling theory and is not developed here.

\begin{remark}[The beta case closes the loop]
For \(X\sim{\rm B}(\theta,1)\) -- exactly the \(V_j\) of
Proposition~\ref{prop:gem} -- both the
range of \(R\) and the point process of record times are Poisson with
intensity \(\theta\,{\rm d}z/z\), the same self-similar Poisson process
underlying the \(\GEM(\theta)\) construction of Section~\ref{sec:ewens}. Every cone or lower-homogeneous family of
Section~\ref{sec:cones} carries the same self-similar structure verbatim, with \(X\)
replaced by that family's induced stick-breaking factor.
\end{remark}

\section{Algebraic aspects: relation to Vershik--Lodkin theory}
\label{sec:vershik-lodkin}

The planar PPP carries two classical graded graphs at once, according to which coordinate of an atom \((t,x)\) we emphasize and which stays concomitant: the Pascal graph of Section~\ref{sec:prelim}, and its factorial-graph counterpart. This section records only the algebraic content of the correspondence; the two graphs themselves, and the adic transformation identifying the choice-operator sweep of Section~\ref{sec:ewens} with a \(\theta\)-weighted deformation of the classical Pascal automorphism, are not re-derived here.

Proposition~\ref{prop:matched} (matched-rate benchmarks) is exactly the finite level-\(k\) occupation law of the Pascal graph at \(k=1\); \cite{VershikEquippedGraphs2015} develops the \((k+1)\)-colour generalisation, in the language of \emph{equipped graded graphs} and projective limits of simplices. The complementary, factorial interpretation is organised instead by rank: the factorial code of a permutation \cite{Laisant1888} records at each step which of the \(n\) available positions a new element occupies, underlying the insertion apparatus of Sections~\ref{sec:insertion}--\ref{sec:record-model}, with the central-measure classification restricted to the sufficient statistic \((\ell,u)\) by Proposition~\ref{prop:record-family} and refined to the full record set by \cite{GnedinBiasRecord}. A different, cycle-based projective system -- deleting the largest element from a permutation's cycle decomposition rather than its sequence of positions -- gives the space of \emph{virtual permutations} \cite{KerovOlshanskiVershik1993}, with central measures classified by \cite{Tsilevich1999}; the Ewens measure, for every \(\theta\ge0\), is the distinguished point of this classification, matching the \(\GEM(\theta)\) construction of Section~\ref{sec:ewens}.

These two structures are not merely analogous: they are the two projections of one bivariate object, already recorded algebraically in Section~\ref{sec:algebra}. The monomial basis \(M_\mu\), adapted to occupancy \(\mu\), is the Pascal-graph interpretation; the fundamental (Gessel) basis \(F_W\), adapted to tail counts \(W\), is the factorial-graph interpretation, since a tail/comparison count is the datum a permutation's record or rank structure carries. Theorem~\ref{thm:general} (the composition law, summed over \(\sigma\in\mathfrak S_k\)) is the statement that a path in the joint bivariate structure projects two ways: forgetting discovery order gives the occupation vector \(\mu\), forgetting magnitude gives the discovery order \(\sigma\) itself. Lemma~\ref{lem:DI} -- the Dirichlet integral over the single ordered chamber \(\Delta_k\) -- is the ``forget order, keep magnitude'' direction; summing over all \(k!\) chambers, as in Theorem~\ref{thm:general}, is the ``forget magnitude, keep order'' direction. The monomial/fundamental duality of quasisymmetric functions is thus the algebraic shadow of the Pascal-graph/factorial-graph duality.

Both bases carry, beyond their finite (probability-measure) boundary theory, a genuinely unbounded extension. An infinite Lévy measure is a semifinite measure: it assigns rates rather than probabilities, finite on most events but unbounded on some. Harmonic functions of this semifinite kind, on the graphs of ordered partitions built from the monomial and fundamental bases above, are classified in \cite{Safonkin2021GK,Safonkin2021Zigzag}, extending the finite (probability) classification of \cite{GnedinRepresentationofCompositionStructures,gnedin2006coherent} to representations of type \(I_\infty\) and \(II_\infty\), not \(II_1\), of the associated AF-algebra. Such semifinite rates are a construction tool, not only a classification target: attaching an independent rate to each side of a hierarchy rather than one scalar builds the two-sided stick-breaking of the \((\theta,\zeta)\) record family of Section~\ref{sec:record-model} directly, and the same device extended coordinatewise builds a multivariate stick-breaking beyond it, not otherwise obtained from a single choice operator.

\end{document}